\documentclass[a4paper,11pt]{article}
\usepackage[latin1]{inputenc}
\usepackage[T1]{fontenc}
\usepackage[british,UKenglish,USenglish,english,american]{babel}
\usepackage{amsmath}
\usepackage{amsfonts}
\usepackage{hyperref}
\usepackage[T1]{fontenc}
\usepackage[latin1]{inputenc}
\usepackage{theorem}
\usepackage{graphics}
\usepackage{makeidx}
\usepackage{fancyhdr}
\usepackage{bbm}
\usepackage{amssymb}
\usepackage{color}
\usepackage{bm}
\numberwithin{equation}{section}
\usepackage{enumerate}
\newtheorem{theorem}{Theorem}[section]

\newtheorem{corollary}[theorem]{Corollary}
\newtheorem{lemma}[theorem]{Lemma}

\newtheorem{definition}[theorem]{Definition}
\theorembodyfont{\rm}

\newtheorem{remark}[theorem]{Remark}

\newcommand{\IR}{{\mathbb{R}}}
\newcommand{\IN}{{\mathbb{N}}}

\renewcommand{\phi}{\varphi}
\renewcommand{\epsilon}{\varepsilon}
\newcommand{\eps}{\varepsilon}

\newcommand{\ltn}{\ensuremath{\left| \! \left| \! \left|}}
\newcommand{\rtn}{\ensuremath{\right| \! \right| \! \right|}}

\newcommand{\aM}{\mathcal{M}}
\newcommand{\aP}{\mathcal{P}}
\newcommand{\aI}{\mathcal{I}}

\newcommand{\aL}{\mathcal{L}}
\newcommand{\oN}{\overline{N}}

\newcommand{\ov}{\overline{v}}
\newcommand{\omegaa}{\,\omega^{(2)}}
\newcommand{\omegaan}{\,\omega^{(2),n}}
\newcommand{\yt}{\widetilde{y}}
\newcommand{\zt}{\widetilde{z}}

\newcommand{\qed}{\hfill $\Box$\smallskip}

\title{Global solutions and random dynamical systems for rough evolution equations}
\author{Robert Hesse\thanks{Institute for Mathematics, Friedrich Schiller University Jena (FSU),  Ernst-Abbe-Platz 2, 07743 Jena, Germany. E-Mail: robert.hesse@uni-jena.de.}~~and~Alexandra 
	Neam\c tu\thanks{Technical University of Munich (TUM), 
		Faculty of Mathematics, 85748 Garching bei M\"unchen, Germany. E-Mail: alexandra.neamtu@tum.de.\newline
		\hspace*{3 mm} The authors are grateful to M.~J.~Garrido-Atienza and B.~Schmalfu\ss{} for helpful comments. AN acknowledges support by a DFG grant in the
		D-A-CH framework (KU 3333/2-1).
}}

\begin{document}
\maketitle
\begin{abstract}
	We consider infinite-dimensional parabolic rough evolution equations. Using regularizing properties of analytic semigroups we prove global-in-time existence of solutions and investigate random dynamical systems for such equations.
\end{abstract}

\section{Introduction}
In this work we analyze global-in-time existence of solutions for rough stochastic partial differential equations (SPDEs)
\begin{align}\label{eq1}
\begin{cases}
d y_{t} = (A  y_{t} + F (y_{t}) )dt + G (y_{t}) d \omega_{t},~~ t\in[0,T] \\
y(0)=\xi.
\end{cases}
\end{align}
Here $T>0$ is a fixed time-horizon, the linear part $A$ is the generator of an analytic $C_{0}$-semigroup $(S(t))_{t\in[0,T]}$ on a separable Banach space $W$ and the initial condition $\xi\in W$. Furthermore we assume that the nonlinearity $F:W\to W$ is Lipschitz and $G:W\to \mathcal{L}(V,W)$ is three times continuously Fr\'echet-differentiable with bounded derivatives. The precise assumptions on the coefficients will be stated in Section~\ref{preliminaries}. Finally, the random input $\omega$ is a Gaussian process which can be lifted to a geometric rough path, for instance a fractional Brownian motion with Hurst parameter $H\in(1/3,1/2]$. We refer to~\cite{GarridoLuSchmalfuss1,HesseNeamtu} for examples of such SPDEs. In order to solve~\eqref{eq1} we rely on the pathwise construction of the rough integral
\begin{align}\label{integral}
\int\limits_{0}^{t} S(t-r)G(y_{r}) d\omega_{r}
\end{align}
developed in~\cite{HesseNeamtu}.
Similar results in this context are available in~\cite{Gubinelli,GubinelliLejayTindel,GubinelliTindel} using rough paths techniques and~\cite{GarridoLuSchmalfuss2} using fractional calculus and more recently in~\cite{HesseNeamtu} using an ansatz which combines these two approaches in a suitable way. As already announced in~\cite{HesseNeamtu} the ultimate goal is to investigate the long-time behavior of~\eqref{eq1} and therefore this work establishes the existence of a pathwise global solution. Consequently, we can show that the solution operator of~\eqref{eq1} generates an infinite-dimensional random dynamical system. \\

Refering to the monograph of Arnold~\cite{Arnold}, it is well-known that an It\^{o}-type stochastic differential equation generates a random dynamical system under natural assumptions on the coefficients.~This fact is based on the flow property, see~\cite{Kunita, Scheutzow}, which can be obtained by Kolmogorov's theorem about the existence of a (H\"older)-continuous random field with finite-dimensional parameter range, i.e.~the parameters of this random field are the time and the non-random initial data.\\

The generation of a random dynamical system from an It\^{o}-type SPDE has been a long-standing open problem, since Kolmogorov's theorem breaks down for random fields parametrized by infinite-dimensional Hilbert spaces,~see~\cite{Mohammed}.~As a consequence it is not trivial how to obtain a random dynamical system from an SPDE, since its solution is defined \emph{almost surely}, which contradicts the \emph{cocycle property}. Particularly, this means that there are exceptional sets which depend on the initial condition and it is not clear how to define a random dynamical system if more than countably many exceptional sets occur.
This problem was fully solved only under very restrictive assumptions on the structure of the noise driving the equation.~For instance if one deals with purely additive noise or multiplicative Stratonovich one, there are standard transformations which reduce the SPDE in a random partial differential equation.~Since this can be solved pathwise it is straightforward to obtain a random dynamical system.~However, for nonlinear multiplicative noise, this technique is no longer applicable, not even if the random input is a Brownian motion.~As a consequence of this issue, dynamical aspects for~\eqref{eq1} such as asymptotic stability, Lyapunov exponents, multiplicative ergodic theorems, random attractors, random invariant manifolds have not been investigated in their full generality.\\

Consequently, a pathwise construction of~\eqref{integral} and implicitly of the solution of~\eqref{eq1} would be the first step to overcome this obstacle. Recently, there has been a growing interest to give a pathwise meaning to the solutions of SPDEs by various techniques, see e.g.~\cite{GubinelliTindel,GarridoLuSchmalfuss2,Hairer}. However there are very few results that explore the pathwise character of the solutions to analyze random dynamical systems and their long-time behavior. Progress in this sense was made for instance in~\cite{GLS, GarridoLuSchmalfuss2} that deal with random dynamical systems for SPDEs driven by a fractional Brownian motion with Hurst parameter $H\in(1/2,1)$ and $H\in(1/3,1/2]$.~Local stability statements can be looked up in~\cite{GS}. Moreover,~\cite{Gao} and~\cite{GarridoLuSchmalfussUnstable} prove random attractors respectively random unstable manifolds for SPDEs driven by a fractional Brownian motion where the range of the Hurst index is $(1/2,1)$. 
All these techniques rely on fractional calculus and require strong assumptions on the coefficients of the SPDEs.\\

To our best knowledge there are very few works that connect the \emph{rough paths}- and \emph{random dynamical systems} perspectives such as~\cite{Gess}.~Here we contribute to this aspect and provide a general framework of random dynamical systems for rough evolution equations under natural/less restrictive assumptions on the coefficients.~The crucial result that opens the door for the random dynamical systems theory is the existence of a global pathwise solution for~\eqref{eq1}. It is known that global-in-time existence of solutions is a challenging question in the context of rough paths techniques, compare~\cite{GarridoLuSchmalfuss2, GubinelliTindel, HesseNeamtu}. This is due to the fact that one obtains certain quadratic estimates on the norms of the solution of~\eqref{eq1}. Hence it is not straightforward if one can extend the local solution on an arbitrary time horizon. Using additional restrictions on the coefficients and of the noisy input~\cite{GarridoLuSchmalfuss2} shows global-in-time existence for~\eqref{eq1} driven by a fractional Brownian motion with Hurst index $H\in(1/3,1/2]$. However, in this work using regularizing properties of analytic $C_{0}$-semigroups, a-priori estimates on certain remainder terms and a standard concatenation procedure we are able to prove using rough paths techniques the global-in-time~existence of solutions. Therefore we succeed in closing the gap in~\cite{HesseNeamtu}.\\

This work is structured as follows. Section~\ref{preliminaries} collects important auxiliary results concerning parabolic evolution equations and rough paths theory. In Section~\ref{sect:sewing} we present a very general Sewing Lemma (Theorem~\ref{lemma_sewing}), which entails the construction of the rough integral~\eqref{integral}. Under suitable assumptions, we are able to derive additional space-regularity of the integral operator, compare~\cite{HesseNeamtu,Gubinelli}, which will turn out to be crucial for the global-in-time existence. 
For the convenience of the reader Section~\ref{local:solution} summarizes basic results regarding the construction of local-in-time solutions for rough evolution equations.~We point out that in the context of rough paths the solution is given by a pair containing the path together with its Gubinelli derivative. These two components satisfy suitable algebraic and analytical properties which are precisely summarized and discussed within Section~\ref{local:solution}.
 These are the main necessary ingredients required in order to comprehend the techniques employed in Section~\ref{sect:global:solution}, where we establish the central result of this paper.~This opens the door to infinite-dimensional random dynamical systems using a rough path approach. Here we only show the existence of a random dynamical system in Section~\ref{sect:rds} and aim to investigate its long-time behavior in future works. 
\section{Preliminaries}\label{preliminaries}
We let $T>0$, $V$ stand for a Hilbert space and $W$ denote a separable Banach space. Furthermore, for any compact interval $J \subset \IR$ we set  
$\Delta_{J} := \left\{\left(t,s \right) \in J^2 \colon t \geq s \right\}$ and 
 $\Delta_{T} := \Delta_{[0,T]}$.~For notational simplicity, if not further stated, we write  $\left| \cdot \right|$ for the norm of an arbitrary Banach space.  Furthermore $C$ denotes a universal constant which varies from line to line. The explicit dependence of $C$ on certain parameters will be precisely stated, whenever required. Finally, we fix $\alpha\in(\frac{1}{3},\frac{1}{2})$. This parameter indicates the H\"older-regularity of the random input. Regarding this we recall the following essential concept in the rough path theory.
\begin{definition}\label{hrp}\emph{($\alpha$-H\"older rough path)}
	Let $J \subset \IR$ be  a compact interval. We call a pair $\bm{\omega}:=(\omega,\omegaa)$ $\alpha$-H\"older rough path if $\omega\in C^{\alpha}(J, V)$ and $\omegaa\in C^{2\alpha}(\Delta_{J}, V\otimes V)$. Furthermore $\omega$ and $\omegaa$ are connected via Chen's relation, meaning that
	\begin{align}\label{chen}
	\omegaa_{ts} - \omegaa_{us} - \omegaa_{tu} = (\omega_{u} - \omega_{s})\otimes ( \omega _{t}- \omega_{u}) 
	,~~ \mbox{for } s,u,t \in J,~~ s \leq u \leq t .
	\end{align}
	In the literature $\omegaa$ is referred to as L\'evy-area or second order process.	
\end{definition}
We further describe an appropriate \emph{distance} between two $\alpha$-H\"older rough paths.
\begin{definition}
	Let $\bm{\omega}$ and $\bm{\widetilde{\omega}}$ be two $\alpha$-H\"older rough paths. We introduce the $\alpha$-H\"older rough path (inhomogeneous) metric
	\begin{align}\label{rp:metric}
	d_{\alpha,J}(\bm{\omega},\widetilde{\bm{\omega}} )
	:= \sup\limits_{(t,s)\in \Delta_J} \frac{|\omega_{t}-\omega_{s}-\widetilde{\omega}_{t}+\widetilde{\omega}_{s}|}{|t-s|^{\alpha}} 
	+ \sup\limits_{(t,s) \in \Delta_{J}}
	\frac{|\omegaa_{ts}-\widetilde{\omega}^{(2)}_{ts}|} {|t-s|^{2\alpha}}.
	\end{align}
	We set $d_{\alpha,T}:=d_{\alpha,[0,T]}$.
\end{definition}
For more details on this topic consult~\cite[Chapter 2]{FritzHairer}. We stress that in our situation we always have that $\omega(0)=0$ and therefore~\eqref{rp:metric} is a metric. We specify concrete examples in Section~\ref{sect:rds}.\\

Having stated the random influences that we consider, we now introduce the assumptions on the linear part and on the coefficients $F$ and $G$. \\

Since we are in the parabolic setting, i.e.~$A$ is a sectorial operator, we can introduce its fractional powers, $(-A)^{\gamma}$ for $\gamma\geq 0$, see \cite[Section 2.6]{Pazy} or \cite{Lunardi}. We denote the domains of the fractional powers of $(-A)$ with $D_{\gamma}$, i.e. $D_{\gamma}:=D((-A)^{\gamma})$ and use following estimates.\\

For $\eta, \kappa\in\mathbb{R}$ we have
\begin{align}
||S(t) ||_{\mathcal{L}(D_{\kappa}, D_{\eta})} =||(-A)^{\eta}S(t)||_{\mathcal{L}(D_{\kappa},W)} \leq C t ^{k-\eta}, &~~\mbox{for  } \eta\geq \kappa \label{hg1}\\
||S(t) - \mbox{Id} ||_{\mathcal{L}(D_{\sigma}, D_{\lambda})} \leq C t^{\sigma-\lambda}, &~~\mbox{for  } \sigma-\lambda\in[0,1]\label{hg2}.
\end{align}
Furthermore, one can show that the following assertions hold true, consult~\cite[Chapter~3]{Pazy}. 
\begin{lemma}
	For any $\nu,\eta,\mu\in[0,1]$, $\kappa,\gamma,\rho\geq 0$ such that $\kappa\leq \gamma+\mu$, there exists a constant $C>0$ such that for $0<q<r<s<t$ we have that
	\begin{align*}
	&||S(t-r) - S(t-q)||_{\mathcal{L}(D_{\kappa},D_{\gamma})} \leq C (r-q)^{\mu}(t-r)^{-\mu-\gamma+\kappa},\\
	&||S(t-r)- S(s-r)- S(t-q) + S(s-q)||_{\mathcal{L}(D_{\rho}, D_{\rho})} \leq C (t-s)^{\eta}(r-q)^{\nu}(s-r)^{-(\nu+\eta)}.
	\end{align*}	
\end{lemma}

For our aims we introduce following function spaces. 
Let $\beta\in(0,1)$ be fixed and let $\overline{W}$ stand for a further Hilbert space. We recall that $C^{\beta}([0,T],W)$ represents the space of $W$-valued H\"older continuous functions on $[0,T]$ and denote by $C^{\alpha}(\Delta_T,\overline{W})$ the space of $\overline{W}$-valued functions on $\Delta_T$ with $z_{t,t}=0$ for all $t\in[0,T]$ and
\begin{align*}
\left\| z \right\|_{\alpha}:= \sup\limits_{0\leq t \leq T} |z_{t0}| + \sup\limits_{0\leq s < t \leq T} \frac{\left|z_{ts}\right|}{(t-s)^{\alpha}}<\infty.
\end{align*}

Furthermore, we define $C^{\beta,\beta}([0,T],W)$ as the space of $W$-valued continuous functions on $[0,T]$ endowed with the norm

\begin{align*}
\left\|y\right\|_{\beta,\beta}
:= \left\|y\right\|_ \infty + \ltn y\rtn_{\beta,\beta}
:= \sup\limits_{0\leq t \leq T } |y_t| + \sup\limits_{0<s<t\leq T }s^{\beta}\frac{|y_t-y_s|}{(t-s)^{\beta}}.
\end{align*}
Similarly we introduce $C^{\alpha+\beta,\beta}(\Delta_T, \overline{W})$ with the norm
\begin{align*}
\left\| z \right\|_{\alpha+\beta,\beta}:=\sup\limits_{0\leq t \leq T } |z_{t0}| + \sup\limits_{0< s < t \leq T} s^{\beta}\frac{\left|z_{ts}\right|}{(t-s)^{\alpha+\beta}}.
\end{align*}
Again $z_{t,t}=0$ for all $t\in [0,T]$.\\

These modified H\"older spaces are well-known in the theory of maximal regularity for parabolic evolution equations, consult~\cite{Lunardi}. These were also used in~\cite{GarridoLuSchmalfuss2}.\\

In this framework we emphasize the following result which will be employed throughout this work. It is well-known that analytic $C_{0}$-semigroups are not H\"older continuous in $0$. However, the following lemma holds true.

\begin{lemma}\label{betabeta}
	Let $\left(S(t)\right)_{t\geq 0}$ be an analytic $C_{0}$-semigroup on $W$. Then we have for all $x \in W$ and all $\beta \in \left[0,1 \right]$ that
	\begin{align*}
	\left\|  S(\cdot) x\right\|_{\beta,\beta}\leq C \left|x\right|,
	\end{align*}
	where $C$ depends only on the semigroup and on $\beta$.
\end{lemma}
\begin{proof}
	\begin{align*}
	\left\| S(\cdot) x\right\|_{\beta,\beta} 
	&= \sup\limits_{0 \leq t \leq T} \left|S(t) x\right| + \sup\limits_{0<s < t \leq T} s^\beta \frac{\left|(S(t)-S(s)) x\right| }{(t-s)^\beta}\\
	& \leq \sup\limits_{0 \leq t \leq T} \left|S(t) x\right| + \sup\limits_{0<s<t\leq T} s^{\beta} \frac{|(S(t-s)-\mbox{Id})S(s)x|}{(t-s)^{\beta}} \\
	&\leq C |x|,
	\end{align*}
	recall (\ref{hg1})
	and (\ref{hg2}).
	\qed \end{proof}\\

This justifies our choice of working with the function space $C^{\beta,\beta}$. Note that if one lets $x\in D_{\beta}$ it suffices to consider only $C^{\beta}$. However, since we analyze random dynamical systems generated by~\eqref{eq1} in $W$~(compare Section~\ref{sect:rds}), we need to take the initial condition $\xi\in W$ instead of $D_{\beta}$.
\newpage
On the coefficients we impose: 

\begin{description}
	\item[(F)] $F \colon W \to W$ is  Lipschitz continuous.
	\item[(G)] $G \colon W \to \aL(V,D_\beta)$ is bounded and three times Frech\'et differentiable with bounded derivatives. Here we demand $\alpha + 2\beta >1$.
\end{description}

\begin{remark}
	\begin{itemize}
		\item [1)] As in~\cite{HesseNeamtu} we set $F\equiv 0$ for simplicity, since this term does not cause additional technical difficulties.
		\item [2)] To our best knowledge $(\textbf{F})$ and $(\textbf{G})$ are the most general assumptions made on the coefficients of the SPDE~\eqref{eq1}, compare~\cite{GarridoLuSchmalfuss2, GubinelliTindel} and the references specified therein.
 
	\end{itemize}
\end{remark}
Finally, we fix some important notations from the rough paths theory, see also~\cite{GubinelliTindel,DeyaGubinelliTindel} and random dynamical systems which will be required later on.\\

\textbf{Notations}: For $y \in C([0,T],W)$ and $z \in C(\Delta_T,\overline{W})$ we set
\begin{align*}
(\delta y)_{ts}&:= y_t-y_s, \\
(\hat\delta y)_{ts}&:= y_t-S(t-s)y_s, \\
(\delta_2 z)_{t \tau s}&:= z_{ts} - z_{t\tau} - z_{\tau s}, \\
(\hat\delta_2 z)_{t \tau s}&:= z_{ts} - z_{t\tau} - S(t-\tau)z_{\tau s}.
\end{align*}
 
  Furthermore we use the notation $\widetilde{\theta}$ in order to indicate the usual shift, namely 
 \begin{align*}
 \widetilde{\theta}_{\tau}y_t &:= y_{t+\tau}, \\
 \widetilde{\theta}_{\tau}z_{ts}&:= z_{t+\tau,s+\tau}.	
 \end{align*}
 The notation $\theta$ always stands for the Wiener shift (this represents an appropriate shift with respect to the noise), more precisely
 \begin{align*}
 \theta_{\tau}\omega_{t}:=\omega_{t+\tau} -\omega_{\tau},
 \end{align*}
 which is explained in detail in Section~\ref{rds}. This is mainly required in the random dynamical systems theory.
 
\section{Sewing Lemma revised}\label{sect:sewing}
In this section we collect concepts from the rough paths theory~\cite{ GubinelliLejayTindel,GubinelliTindel} and recall some important results regarding the construction and properties of~\eqref{integral}. For further details and complete proofs of the following statements, consult~\cite[Section 4]{HesseNeamtu}. A key point in this framework is given by the Sewing Lemma~\cite{GubinelliTindel}. This ensures the existence of a rough integral under suitable assumptions.
Here we use a special case of the Sewing Lemma proved in~\cite{HesseNeamtu}. Based on this we develop a more general statement which is crucial for Section~\ref{sect:global:solution}. 

\begin{theorem}[Sewing Lemma, Theorem~4.1~\cite{HesseNeamtu}]\label{lemma_sewing}
	Let $W$ be a separable Banach space and $(S(t))_{t\geq 0}$ be an analytic $C_{0}$-semigroup on $W$.
	Furthermore, let $\Xi \in C \! \left( \Delta_{T} , W\right)$ be an approximation term satisfying the following properties
	for all $0 \leq u \leq m \leq v \leq T$ :
	\begin{align}
	\left|\Xi_{vu}\right| &\leq c_1 \left(v-u \right)^{\alpha},\label{assumption_Xi}\\
	\left|\left(\hat\delta_2 \Xi\right)_{vmu} \right| & \leq c_2 \left(v-u \right)^{\rho}.\label{assumption_deltaXi1}
	\end{align}
	Here we impose $0 \leq\alpha\leq 1$ and $\rho>1$.\\			
	
	Then there exists a unique $\aI\Xi \in C\!\left(\left[0,T\right],W \right)$, such that 
	\begin{align}
	\aI\Xi_0 &=0, \\
	\left|\left(\hat\delta \aI \Xi \right)_{ts} \right|&\leq C \left(c_1+c_2 \right) \left(t-s \right)^{\alpha} 
	\label{property_deltaIXiest}\\
	\left|(\hat\delta \aI \Xi)_{ts}-\Xi_{ts} \right|&\leq C c_2 \left(t-s \right)^{\rho}.
	\label{property_IXiest1}
	\end{align}
\end{theorem}

In order to interpret $\aI\Xi$ as a rough integral it is crucial that this fulfills integral-like properties, namely it has to be given by a limit of finite sums and satisfy a shift property. These have been rigorously verified in~\cite[Section~4]{HesseNeamtu}.
\begin{corollary}[Approximation by finite sums, Corollary~4.3~\cite{HesseNeamtu}] 
	\label{corollary_sewing_integral}
	Under the assumptions of Theorem~\ref{lemma_sewing} it holds that
	\begin{align}
	\left(\hat\delta \aI \Xi \right)_{ts} 
	&= \lim\limits_{\left|\aP\right| \to 0} \sum\limits_{\left[u,v\right] \in \aP} {S(t-v) \Xi_{vu}},
	\label{property_IXieq}
	\end{align}
	where $\left|\aP\right|$ stands for the mesh of the given partition $\aP=\aP(s,t)$.
\end{corollary}
\begin{remark}
	Corollary~\ref{corollary_sewing_integral} implies the additivity of the rough integral.
\end{remark}
In order to introduce the shift property of the rough integral $\mathcal{I}\Xi$ we recall that for $\tau>0$
\begin{align*}
\widetilde{\theta}_\tau \Xi_{vu} = \Xi_{v+\tau,u+\tau},
\end{align*}
see Section~\ref{preliminaries}.
Considering this, one can easily verify the shift property of $\mathcal{I}\Xi$.
\begin{lemma}[Shift property, Corollary~4.5~\cite{HesseNeamtu}]\label{lemma_sewing_shift}
	Under the assumptions of Theorem~\ref{lemma_sewing} we have 	\begin{align*}
	(\hat\delta \aI \Xi)_{ts} = (\hat\delta \aI \widetilde{\theta}_{\tau} \Xi)_{t-\tau,s-\tau}, ~~\mbox{for } \tau \leq s \leq t.
	\end{align*}
\end{lemma}

In order to obtain a global solution for~\eqref{eq1} we have to precisely analyze the spatial regularity of $\aI\Xi$. To this aim we formulate the main result of this section.
\begin{corollary}\label{corollary_sewing_epsestimate_adapted}
	Additionally to the restrictions of Theorem~\ref{lemma_sewing} we further assume that
	\begin{align*}
	\left|S(v-u) \Xi_{vu} \right|_{D_\eps} \leq c_1' \left(v-u \right)^{\alpha'}
	\end{align*}
	where $0\leq\alpha'\leq 1$ and $0\leq\varepsilon<1$. Then we have
	\begin{align*}
	\left|(\hat\delta\aI\Xi)_{ts} \right|_{D_\eps} \leq C \left(c_1' (t-s)^{\alpha'} + c_2 (t-s)^{\rho-\eps}\right).
	\end{align*}
\end{corollary}
\begin{proof}
	The computation is similar to~\cite[Corollary~4.6]{HesseNeamtu}. We define $\aP_n$ as the $n$-th dyadic partition of $[s,t]$ and set
	\begin{align*}
	N^n_{ts} &:= \sum\limits_{\left[u,v\right]\in \aP_n} S(t-v) \Xi_{vu},\\
	\oN^n_{ts} &:= \sum\limits_{\substack{\left[u,v\right]\in \aP_n\\ v \neq t}} S(t-v) \Xi_{vu}, \\
	\overline{v}_n &:= \max\left\{v < t \colon [u,v] \in \aP_n \right\}.
	\end{align*}
	By standard computation we get
	\begin{align*}
	\oN^n_{ts} - \oN^{n+1}_{ts} 
	=\sum\limits_{\substack{\left[u,v\right]\in \aP_n \\ v \neq t}} S(t-v) (\hat\delta_2 \Xi)_{vmu} - S(t-\ov_{n+1}) \Xi_{\ov_{n+1}\ov_n}. 
	\end{align*}
	Hence we obtain
	\begin{align*}
	\left|\oN^n_{ts} - \oN^{n+1}_{ts} \right|_{D_\eps}
	&\leq \left|S(t-\ov_{n+1}) \Xi_{\ov_{n+1}\ov_n}\right|_{D_\eps}
	+ \sum\limits_{\substack{\left[u,v\right]\in \aP_n \\ v \neq t}} \left| S(t-v) (\hat\delta_2 \Xi)_{vmu} \right|_{D_\eps} .
	\end{align*}
	Note that $t-\ov_{n+1}=\ov_{n+1}-\ov_n$ since we consider a dyadic partition. Regarding our assumptions, this further results in
	\begin{align*}
	\left|\oN^n_{ts} - \oN^{n+1}_{ts} \right|_{D_\eps} 
	&\leq c_1' (\ov_{n+1}-\ov_n)^{\alpha'} 
	+ C c_2 \sum\limits_{\substack{\left[u,v\right]\in \aP_n \\ v \neq t}} \left(t-v \right)^{-\eps} \left(v-u\right)^{\rho} \\
	&\leq C \left(c_1' (t-s)^{\alpha'} 2^{-n\alpha'} + c_2 (t-s)^{\rho-\eps} 2^{-n(\rho-1)} \right).
	\end{align*}
	Consequently, the previous expression is summable and yields that $\oN^n_{ts} \to \oN_{ts}$ (in $D_{\varepsilon}$) for all $0 \leq s <t \leq T$.
	Corollary \ref{corollary_sewing_integral} entails
	\begin{align*}
	\left|(\hat{\delta}\aI\Xi)_{ts}-\oN_{ts} \right| = \lim\limits_{n \to \infty} \left|N^n_{ts}-\oN^n_{ts} \right|
	= \lim\limits_{n \to \infty}\left|\Xi_{t \ov_n} \right|
	\stackrel{\eqref{assumption_Xi}}{\leq} 
	c_1 \lim\limits_{n \to \infty} \left(t-\ov_{n+1} \right)^\alpha =0.
	\end{align*}
	Hence, we know that $\oN \equiv (\hat{\delta}\aI\Xi)$ and finally infer that
	\begin{align*}
	\left|(\hat\delta \aI \Xi)_{ts} \right|_{D_{\eps}} 
	= \left|\oN_{ts} \right|_{D_{\eps}} 
	\leq C \left(c_1' (t-s)^{\alpha'} + c_2 (t-s)^{\rho-\eps} \right).
	\end{align*}
	This proves the statement.
	\qed 
\end{proof}\\

As already emphasized the spacial regularity of $\mathcal{I}\Xi$ is essential for the computation in Section~\ref{sect:global:solution}.

\section{Solution theory for rough SPDEs}\label{local:solution}
 For a better comprehension of Section~\ref{sect:global:solution} we point out certain results regarding the existence and uniqueness of a local solution for~\eqref{eq1}. Here we consider another approach than in~\cite{HesseNeamtu, GubinelliTindel} which finally enables us to obtain a global-in-time solution. 
 
 \begin{remark}
 	In this setting $V\otimes V $ denotes the usual tensor product of Hilbert spaces. If one wishes to work in Banach spaces, then one should consider the projective tensor product, since the property
 	\begin{align*}
 	\mathcal{L}(V,\mathcal{L}(V,W))\hookrightarrow \mathcal{L} (V\otimes V, W)
 	\end{align*} 
 	is required.
 	This is known to hold true, consult Theorem 2.9 in \cite{Ryan}.
 	In the following, for notational simplicity we drop the tensor symbol. 
 \end{remark}

We firstly indicate a heuristic computation which is required in order to construct the rough integral~\eqref{integral} and therefore to give a pathwise meaning to the solution~of~\eqref{eq1}.~These deliberations are rigorously justified in~\cite{HesseNeamtu}, which essentially combines the techniques in~\cite{GubinelliTindel} and~\cite{GarridoLuSchmalfuss2}. Here we only want to provide the general intuition of how the solution of~\eqref{eq1} should look like and focus on its global existence.\\

The strategy to define~\eqref{integral} relies on an approximation procedure. We firstly consider a smooth path $\omega$ and a continuous trajectory $y$. The general argument eventually follows considering smooth approximations of $\omega$, as shortly indicated below. Our aim is to define~\eqref{integral} using Riemann-Stieltjes sums and a Taylor expansion for $G$. By a formal computation, this reads as

\begin{align*}
\int\limits_{0}^{t}{S(t-r)G(y_r)}d\omega_r 
= &\sum\limits_{[u,v]\in \aP} S(t-v) \int\limits_{u}^{v}{S(v-r)G(y_r)}d\omega_r \\
\approx &\sum\limits_{[u,v]\in \aP} S(t-v) \Big[\int\limits_{u}^{v}{S(v-r)G(y_u)} d\omega_r 
+ \int\limits_{u}^{v}{S(v-r)DG(y_u)(y_r-y_u)} d\omega_r\Big] \\
=: &\sum\limits_{[u,v]\in \aP} S(t-v) 
\Big[\omega_{vu}^S(G(y_u)) +z_{vu}(DG(y_u))\Big].
\end{align*}
Here we introduced the notation
\begin{align}\label{omegas_h}
\omega^{S}_{vu}(G(y_{u})):=\int\limits_{u}^{v} S(v-r) G (y_{u}) d\omega_{r},
\end{align}
respectively

\begin{align}\label{z_h}
z_{vu}(DG(y_{u})):= \int\limits_{u}^{v} S(v-r)  DG(y_{u}) (\delta y)_{ru}d\omega_{r}.
\end{align}
By a classical integration by parts formula, see Theorem~3.5~in~\cite{Pazy} one can argue that the term $\omega^{S}$ can be defined for a rough input $\omega$. However, we have to continue our deliberations to obtain a meaningful definition of $z$. To this aim we let $E \in \aL(W; \aL(V;W))$ denote a placeholder which stands for $DG$ and consider further Riemann-Stieltjes sums for~\eqref{z_h}. Namely, for a partition $\aP=\aP([s,t])$ we have
\begin{align*}
z_{ts}(E) 
&= \int\limits_{s}^{t}{S(t-r)E (y_r-y_s)} d\omega_r 
=\sum\limits_{[u,v]\in \aP} S(t-v) \int\limits_{u}^{v}{S(v-r)E (y_r-y_s)} d\omega_r \\
= &\sum\limits_{[u,v]\in \aP} S(t-v) \Big[\int\limits_{u}^{v}{S(v-r)E (y_r-S(r-u)y_u)} d\omega_r \\
	&+\int\limits_{u}^{v}{S(v-r)E S(r-u)y_u} d\omega_r
	+\int\limits_{u}^{v}{S(v-r)E y_s} d\omega_r
	\Big],
\end{align*}
where in the second step we subtract the expression $S(r-u)y_{u}$. Regarding this we make following ansatz for the first term of the previous expression. Since $y$ is supposed to solve~\eqref{eq1}, this should satisfy the variation of constants formula
\begin{align*}
y_r-S(r-u)y_u = \int\limits_{u}^{r} {S(r-q)G(y_q)} d\omega_q.
\end{align*}
Plugging this in the expression of $z$, we immediately obtain 
\begin{align*}
z_{ts}(E)  
&=: \sum\limits_{[u,v]\in \aP} S(t-v) \Big[\int\limits_{u}^{v}{S(v-r)E \int\limits_{u}^{r}{S(r-q) G(y_q)}d\omega_q} d\omega_r +a_{vu}(E,y_u)\Big]-\omega^S_{ts}(Ey_s)\\
&\approx \sum\limits_{[u,v]\in \aP} S(t-v) \Big[\int\limits_{u}^{v}{S(v-r)E \int\limits_{u}^{r}{S(r-q) G(y_u)}d\omega_q} d\omega_r +a_{vu}(E,y_u)\Big]-\omega^S_{ts}(Ey_s)
	\\
	=: &\sum\limits_{[u,v]\in \aP} S(t-v) \Big[ b_{vu}(E,G(y_u)) + a_{vu}(E,y_u)\Big] - \omega^S_{ts}(E y_s).
\end{align*}
For simplicity we introduced the notation
\begin{align*}
b_{vu}(E,G(y_{u})) :=\int\limits_{u}^{v} S(v-r) E \int\limits_{u}^{r} S(r-q) G (y_{u}) d\omega_{q}d\omega_{r},
\end{align*}
respectively
\begin{align*}
a_{vu}(E,y_{u}):=\int\limits_{u}^{v} S(v-r) E S (r-u) 
y_{u} d\omega_{r}.
\end{align*}
	This indicates that we have to define $a$, $b$ 
and $\omega^{S}$ in order to fully characterize $z$. At the very first sight, it is not straightforward how to introduce $b$. Therefore we continue our heuristic computation. We let $K$ denote a placeholder which stands for $G$. Again, using Riemann-Stieltjes sums and a suitable approximation of certain terms below we infer that
\begin{align*}
\int\limits_{s}^{t}{S(t-r) E \int\limits_{s}^{r}{S(r-q)K}d\omega_q} d\omega_r 
= &\sum\limits_{[u,v]\in \aP} S(t-v) \int\limits_{u}^{v}{S(v-r) E \int\limits_{s}^{r}{S(r-q)K}d\omega_q} d\omega_r \\
= &\sum\limits_{[u,v]\in \aP} S(t-v) \int\limits_{u}^{v}{S(v-r) E \int\limits_{s}^{u}{S(r-q)K}d\omega_q} d\omega_r \\
&+ \sum\limits_{[u,v]\in \aP} S(t-v) \int\limits_{u}^{v}{S(v-r) E \int\limits_{u}^{r}{S(r-q)K}d\omega_q} d\omega_r
\end{align*}
\begin{align*}
\phantom{
\int\limits_{s}^{t}{S(t-r) E \int\limits_{s}^{r}{S(r-q)K}d\omega_q} d\omega_r}
\approx & \sum\limits_{[u,v]\in \aP} S(t-v) \int\limits_{u}^{v}{S(v-r) E \int\limits_{s}^{u}{S(u-q)K}d\omega_q} d\omega_r \\
&+ \sum\limits_{[u,v]\in \aP} S(t-v) \int\limits_{u}^{v}{S(v-r) E \int\limits_{u}^{r}{K}d\omega_q} d\omega_r \\ 
= :& \sum\limits_{[u,v]\in \aP} S(t-v) \left[\omega^S_{vu} (E \omega^S_{us}(K)) + c_{vu}(E,K) \right].
\end{align*} 
Here
\begin{align*}
c_{ts}(E,K):=\int\limits_{s}^{t} S(t-r) E K (\omega_{r}-\omega_{s}) d\omega_{r}.
\end{align*}

 Motivated by this heuristic computation, one can introduce similar to~\cite{GarridoLuSchmalfuss1, HesseNeamtu} these processes for smooth paths $(\omega^{n},\omegaan)$ approximating $(\omega,\omega^{(2)})$ in the  $d_{\alpha, T}$-metric. Here \begin{align*}
 \omegaan_{ts}:=\int\limits_{s}^{t} (\delta \omega^{n})_{rs} \otimes d\omega^{n}_r.
 \end{align*}
 Thereafter, the passage to the limit entails a suitable construction/interpretation of all these expressions according to~\cite[Section~5]{HesseNeamtu}. More precisely, the following results hold true.
 
	\begin{lemma}\label{lemma:cont:dependence1} We have that
	\begin{align}
	&\omega^{S,n} \to \omega^S \text{ in } C^{\alpha}\left(\left[0,T\right],\aL(\aL(V,W),W) \right) \\
	&a^n \to a \text{ in } C^{\alpha}\left(\left[0,T\right],\aL(\aL(W \otimes V,W) \times W,W) \right) \label{convergence_a}\\
	&c^n \to c \text{ in } C^{2\alpha}\left(\left[0,T\right],\aL(\aL(W \otimes V,W) \times \aL(V,W),W) \right)\\
	&	b^n \to b \text{ in } C^{2\alpha}\left(\left[0,T\right],\aL(\aL(W \otimes V,W) \times \aL(V,D_\beta),W) \right)\label{convergence_b}.
	\end{align}
\end{lemma}

 We collect further results which are essential for the computation in Section~\ref{sect:global:solution}.

  \begin{lemma}
  	Let $K\in\mathcal{L}(V,W)$, $E\in\mathcal{L}(W\otimes V, W) $ and
  	 $(\omega,\omegaa)$ be an $\alpha$-H\"older rough path. Then $\omega^{S}$, $a$ and $c$ can be defined using integration by parts as
  	\begin{align}
  	\label{definition_omegaS}
  	\omega^{S}_{ts}(K) &= S(t-s) K (\delta \omega)_{ts} -A \int\limits_{s}^{t}{S(t-r)K(\delta \omega)_{tr}} dr, \\
  	\label{definition_a}
  	a_{ts}(E,x) &= \omega^{S}_{ts}(Ex) + \int\limits_{s}^{t}{\omega^{S}_{tr}\left(E A S(r-s) x \right)} dr, \\
  	\label{definition_c}
  	c_{ts}(E,K)&= \omega^{S}_{ts} (EK \left(\delta \omega \right)_{ts}) 
  	- S(t-s) EK \omegaa_{ts}- \int\limits_{s}^{t} A S(t-r) EK \omegaa_{tr} dr.
  	\end{align}
  \end{lemma}
  
  As precisely stated in~\cite[Section~5]{HesseNeamtu} these processes satisfy important analytic and algebraic properties which perfectly fit in the rough path framework. We shortly indicate them together with a generalization which will be required in Section~\ref{sect:global:solution}.
  \begin{lemma}[Analytic properties, Lemma~5.4 and Lemma~5.11~\cite{HesseNeamtu}]\label{supporting_analytic}
  The following analytic estimates hold true: 
  	\begin{align}
  	&\left|\omega^{S}_{ts}(K) \right| \leq C \ltn \omega \rtn_{\alpha} \left| K \right| \left(t-s \right)^\alpha, \label{estimate_omegaS}\\
  	&\left| a_{ts}(E,x) \right| \leq C \ltn \omega \rtn_{\alpha} \left| E \right| |x|_{W} \left(t-s \right)^{\alpha}, ~~\mbox{for } x\in W\label{estimate_a},\\
  	&\left|a_{ts}(E,x) - \omega^{S}_{ts}(Ex) \right| \leq C \ltn \omega \rtn_{\alpha} \left| E \right| \left|x \right|_{D_\gamma} (t-s)^{\alpha+\gamma}, ~~\mbox{for } x\in D_{\gamma} \mbox{ and } 0< \gamma\leq 1,\label{estimate_a2}\\
  	&\left|c_{ts}(E,K) \right| \leq C \left(\ltn \omega \rtn_{\alpha}+\left\|\omegaa \right\|_{2\alpha} \right) \left|E \right| \left| K \right| \left(t-s \right)^{2\alpha},\label{estimate_c}\\
  	&\left|b_{ts}(E,K) \right| \leq C \left| E \right| \left| K \right|_{\mathcal{L}(V,D_{\beta})} \left(\ltn \omega\rtn^2_\alpha + \left\| \omegaa\right\|_{2\alpha} \right) \left(t-s \right)^{2\alpha}. \label{estimate_b}
  	\end{align}
  \end{lemma}
  The next statement gives an extension of~\eqref{estimate_omegaS}.
  \begin{lemma}\label{omega_s_beta}
  	Let $K \in \aL(V,D_\gamma)$ for $0 \leq \gamma \leq 1 $. Then
  	\begin{align}
  	\left|\omega^{S}_{ts}(K) \right|_{D_\gamma} 
  	\leq C \ltn \omega \rtn_{\alpha} \left| K \right|_{\aL(V,D_\gamma)} \left(t-s \right)^\alpha \label{estimate_omegaS2}.
  	\end{align}
  \end{lemma}
  \begin{proof}
  	The proof can immediately be derived using~\eqref{definition_omegaS}, see Lemma 5.4~in~\cite{HesseNeamtu} for a detailed computation.	
  	\qed 
  \end{proof}
  \begin{lemma}[Algebraic properties, Lemma~5.3 and Lemma~5.11~\cite{HesseNeamtu}]\label{supporting_algebraic}
  	The algebraic relations are satisfied:
  	\begin{align}
  	&(\hat\delta_{2} \omega^{S})_{t \tau s}(K) = 0, \label{algebraic_omegaS} \\
  	&(\hat\delta_{2} a)_{t \tau s}(E,x) = a_{t \tau} (E,(S(\tau-s)-id)x),\label{algebraic_a}\\
  	&(\hat\delta_{2} c)_{t \tau s}(E,K) = \omega^{S}_{t \tau}(EK (\delta \omega)_{\tau s}), \label{algebraic_c} \\
  	&(\hat\delta_{2} b)_{t \tau s}(E,K) = a_{t\tau} (E,\omega^S_{\tau s}(K)).\label{algebraic_b}
  	\end{align} 
  \end{lemma}

     Putting all arguments from our heuristic computation together, we immediately observe that the role of the abstract approximation term in Theorem~\ref{lemma_sewing} is represented by 
     \begin{align*}
     \Xi_{vu}^{(y)} = \Xi_{vu}^{(y)}(y,z)= \omega^S_{vu}(G(y_u)) + z_{vu}(DG(y_u)).
     \end{align*}
     Verifying~\eqref{assumption_Xi} and~\eqref{assumption_deltaXi1} the existence of the rough integral $\mathcal{I}\Xi^{(y)}$ is obtained. The same holds true for
     \begin{align*}
     \Xi^{(z)}(y,y)_{vu}(E) = b_{vu}(E,G(y_u)) + a_{vu}(E,y_u).
     \end{align*}
Regarding this one concludes that the solution of~\eqref{eq1} has the structure 
  \begin{align}
  y_{t} &= S(t) \xi + \aI \Xi^{(y)}(y,z)_{t}\nonumber\\
  & = S(t) \xi + \lim\limits_{|\aP([0,t])|\to 0}\sum\limits_{[u,v]\in\aP([0,t])} S(t-v) [\omega^{S}_{vu}(G(y_{u})) + z_{vu} (DG(y_{u}))  ]\\
  z_{ts}(E)&= (\hat\delta \aI \Xi^{(z)}(y,y))_{ts}(E) - \omega^S_{ts}(E y_s) \nonumber \\
  &=\lim\limits_{|\aP([s,t])| \to 0 }\sum\limits_{[u,v]\in\aP([s,t])} S(t-v) [b_{vu}(E,G(y_{u})) + a_{vu}(E,y_{u})] -\omega^{S}_{ts}(Ey_{s}) .
  \end{align}
  One can show that~\eqref{eq1} has a unique local-in-time solution in a suitable function space which incorporates the algebraic and analytic properties of the pair $(y,z)$. For further details see~\cite[Section~6]{HesseNeamtu}. For a better comprehension we briefly indicate this framework and point out the local-in-time existence result.\\
  
  For a pair $(y,z)$ where $(y_t)_{t \in \left[0,T\right]}$ is a $W$-valued path and the area term $(z_{ts})_{(t,s)\in \Delta_T}$,
  $ z_{ts} \in \aL(\aL(W \otimes V,W),W)
  $, we consider the function space
  \begin{align*}
  X_{\omega,T}:=\Big\{(y,z): & ~y \in C^{\beta,\beta}\left([0,T],W\right), \\
  &~ z \in C^{\alpha}\left(\Delta_T,\aL(\aL(W \otimes V,W),W) \right) 
  \cap C^{\alpha+\beta,\beta}\left(\Delta_T,\aL(\aL(W \otimes V,W),W) \right),\\
  &(\hat\delta_2 z)_{t \tau s} = \omega^S_{t\tau}(\cdot (\delta y)_{\tau s}) \Big\},
  \end{align*}
  endowed with norm
  \begin{align}\label{norm1}
  \left\|(y,z) \right\|_{X} := \left\|y \right\|_\infty + \ltn y\rtn_{\beta,\beta} 
  + \sup\limits_{0 \leq s ,  t \leq T} \frac{\left|z_{ts} \right|}{(t-s)^\alpha}
  +\sup\limits_{0 < s ,  t \leq T} s^\beta \frac{\left|z_{ts} \right|}{(t-s)^{\alpha+\beta}}.
  \end{align}
Since we also have to incorporate the initial condition $\xi$ we introduce the mapping
  \begin{align*}
  \aM_{T,\omega,\xi} \colon X_{\omega,T} \to X_{\omega,T} \quad \quad \quad
  \aM_{T,\omega,\xi}(y,z) = (\yt,\zt),
  \end{align*} 
  with
  \begin{align*}
  \yt_t &:= S(t) \xi + \aI\Xi^{(y)}(y,z)_t, \\
  \zt_{ts}(E) &:= \big(\hat\delta\aI\Xi^{(z)}(y,\yt) \big)_{ts}(E)- \omega_{ts}^S(E\yt_s),
  \end{align*}
  where
  \begin{align*}
  \Xi^{(y)}_{vu}=\Xi^{(y)}(y,z)_{vu}
  &:= \omega^S_{vu}(G(y_u)) + z_{vu}(DG(y_u)), \\
  \Xi^{(z)}_{vu}(E) = \Xi^{(z)}(y,\widetilde{y})_{vu}(E) 
  &:= b_{vu}(E,G(y_u)) + a_{vu}(E,\widetilde{y}_u).
  \end{align*}
 The reason why we choose to work with these modified H\"older spaces is given by the fact that the analytic $C_{0}$-semigroup $(S(t))_{t\in[0,T]}$ is not H\"older-continuous in $0$ but belongs to $C^{\beta,\beta}$, recall Lemma~\ref{betabeta}.\\


In this setting one has the following existence result for~\eqref{eq1}.
\begin{theorem}(Theorem~6.9~\cite{HesseNeamtu})\label{theorem_localsolution}
Let $\frac{1}{3} < \alpha \leq \frac{1}{2}$ and $0 < \beta < \alpha$ with $\alpha+2\beta>1$.
Furthermore, choose $r>1$ with $|\xi|\leq r$. 
Then there exist a $T=T(\omega,r)>0$ such that the mapping
$\aM_{T,\omega,\xi}$ has a unique fixed-point $(y,z)\in X_{\omega,T}$. Moreover, this satisfies the estimates
\begin{align}
\left\|y_{T}\right\|_{D_{\beta}} 
& \leq C \left(r~T^{-\beta}+\left\|(y,z) \right\|_{X}^2 T^{\alpha-\beta}\right),
\label{estimate_y_beta} \\ 
\left\|\aM_T(y,z)\right\|_{X} 
&\leq C \left( r + \left(1+\left\|(y,z) \right\|_X^2\right)T^\alpha \right)  \label{estimate_aM}.
\end{align}
If additionally $\xi \in D_\beta$, then $y \in C^\beta\left([0,T],W\right)$ and $z \in C^{\alpha+\beta}\left(\Delta_T,\aL(\aL(W \otimes V,W),W) \right)$.
\end{theorem}	
\begin{proof}
The uniqueness of a local solution for~\eqref{eq1} follows by Theorem 6.13~in~\cite{HesseNeamtu}. For further details consult~\cite[Section~6]{HesseNeamtu}.
\qed 
\end{proof}\\
Furthermore we provide two results regarding the fixed-points of $\mathcal{M}_{T,\omega,\xi}$ which can be obtained by a straightforward computation. 
\begin{remark}\label{remark_localsolution_cut}
	If $(y,z)$ is a fixed-point of $\aM_{T,\omega ,\xi}$ than for any $\tilde{T}< T$ the restriction of $(y,z)$ on $[0,\tilde{T}] \times \Delta_{\tilde{T}}$ is a fixed-point of $\aM_{\tilde{T},\omega ,\xi}$. 
\end{remark}
We point out the following essential facts for the random dynamical systems theory, see Section~\ref{rds}.
\begin{remark}
	Note that if $(\omega,\omegaa)$ is an $\alpha$-H\"older rough path, then the shift $(\theta_{\tau}\omega,\widetilde{\theta}_{\tau}\omegaa)$ is again an $\alpha$-H\"older rough path. The proof is conducted in Lemma~\ref{shift:rp}. From this fact one can infer that the $\widetilde{\theta}_\tau$-shift of all the supporting processes accordingly depends on $\theta_{\tau}\omega$ respectively $\widetilde{\theta}_{\tau}\omegaa$.
\end{remark}

Keeping this in mind we state:
\begin{lemma}\label{lemma_solution_shift} Let $T>0$ and $(y,z) \in X_{\omega,T}$ be a fixed-point of $\aM_{T,\omega ,\xi}$. Then for any $\tau \in [0,T)$ there exists a fixed-point of $\aM_{T-\tau,\theta_{\tau}\omega ,y_{\tau}}$ given by
$(\widetilde{\theta}_{\tau}y, \widetilde{\theta}_{\tau}z)$.
\end{lemma}
\begin{proof}
The proof follows the lines of Lemma~6.11~in~\cite{HesseNeamtu}. We use the notation $\Xi^{(y/z)}_{\omega}$ and $\Xi^{(y/z)}_{{\theta}_{\cdot}\omega}$ in order to emphasize the shifts with respect to $\omega$. By standard computations we get
\begin{align*}
\widetilde{\theta}_{\tau}y_t &= y_{t+\tau} = S(t+\tau) \xi + \aI\Xi_{\omega}^{(y)}(y,z)_{t+\tau} \\
&= S(t) y_{\tau} + (\hat\delta \aI\Xi_\omega^{(y)}(y,z))_{t+\tau,\tau} \\
&= S(t) y_{\tau} +  \aI\Xi_{\theta_{\tau}\omega}^{(y)}(\widetilde{\theta}_{\tau}y,\widetilde{\theta}_{\tau}z)_{t}.
\end{align*}
Furthermore,
\begin{align*}
\widetilde{\theta}_{\tau}z_{ts}(E) =z_{t+\tau,s+\tau}(E)
&= (\hat\delta \aI \Xi_{\omega}^{(z)}(y,y))_{t+\tau,s+\tau} (E) - \omega^S_{t+\tau,s+\tau}(E y_{s+\tau})\\
&= (\hat\delta \aI \Xi_{{\theta}_{\tau}\omega}^{(z)}(\widetilde{\theta}_{\tau}y,\widetilde{\theta}_{\tau}y))_{ts} (E) - \widetilde{\theta}_{\tau}\omega^S_{ts}(E \widetilde{\theta}_{\tau}y_{s}).
\end{align*}	
	\qed
	\end{proof}\\
The deliberations conducted in Section~\ref{sect:global:solution}
improve these results.
\section{Construction of the global-in-time solution}\label{sect:global:solution}

As recalled in the previous section, working with~\eqref{norm1} leads to quadratic estimates for the norm of $(y,z)$ in $X_{\omega,T}$. From this approach it is not clear how/if one can extend the unique local solution on an arbitrary time horizon. Therefore we need different arguments for the global-in-time existence. To this aim, similar to the finite-dimensional case, see~\cite[Section~8.5]{FritzHairer}, it is convenient to work with the norm of certain remainder terms, which is common in the rough paths theory.
 
\begin{definition}
 Let $(y,z) \in X_{\omega,T}$.
  Then we define the remainders
\begin{align*}
R^Y_{ts}&:=(\hat\delta y)_{ts} - \omega^S_{ts}(G(y_s)), \\
R^Z_{ts}(E) &:= z_{ts}(E) - b_{ts}(E,G(y_s)).
\end{align*}
\end{definition}
\begin{remark}
If $S=\mbox{Id}$ and $(y,z)$ is a fixed-point of $\aM$, then the previous terms read as
\begin{align*}
R^{Y}_{ts} =(\delta y)_{ts} - G(y_{s})(\delta\omega)_{ts},
\end{align*}
respectively
\begin{align*}
R^{Z}_{ts}(E)= E\int\limits_{s}^{t} R^{Y}_{rs}d\omega_{r}.
\end{align*}
The expression for the remainder $R^{Y}$ is the same as the one in the finite-dimensional case, compare~\cite[Section~8.5]{FritzHairer}. In contrast to the finite-dimensional setting, $R^{Z}$ is required here to estimate the quadratic terms appearing in~\eqref{estimate_aM}.
\end{remark}

\begin{definition}
Let $(y,z)\in X_{\omega,T}$ such that  
\begin{align}\label{fi}
\Phi_T(y,z) := \left\|y \right\|_{\infty,D_{2\beta},T} + \left\| R^Y \right\|_{2\beta,T} + \left\|R^Z \right\|_{\alpha+2\beta,T}<\infty,
\end{align}
where
\begin{align*}
\left\| R^Y \right\|_{2\beta,T} &:= \sup\limits_{0\leq s<t\leq T} \frac{\left|R^Y_{ts} \right|}{(t-s)^{2\beta}}, \\
\left\| R^Z \right\|_{\alpha + 2\beta,T} &:= \sup\limits_{0\leq s<t\leq T} \sup\limits_{\left|E \right|\leq 1} 
		\frac{\left|R^Z_{ts}(E) \right|}{(t-s)^{\alpha + 2\beta}}.\\
\end{align*}
The space of all pairs $(y,z)$ satisfying~\eqref{fi} is denoted by $\widetilde{X}_{\omega,T}$.
\end{definition}
This leads to the next result.
\begin{lemma}
Let $(y,z) \in \widetilde{X}_{\omega,T}$. Then we obtain the following estimates
\begin{align}
\ltn y \rtn_{\beta,T} &\leq C \left(T^{\alpha-\beta} + T^\beta \Phi_T(y,z) \right), \label{estimate_y_AT} \\
\left\|z \right\|_{\alpha+\beta,T} &\leq C \left(T^{\alpha-\beta} + T^\beta \Phi_T(y,z) \right). \label{estimate_z_AT}
\end{align}
\end{lemma}
\begin{proof}
Regarding the definition of $R^{Y}$, $\hat{\delta}y$ and Lemma \ref{omega_s_beta}
\begin{align*}
\ltn y \rtn_{\beta,T} 
&= \sup\limits_{0\leq s< t \leq T} \frac{\left|(\delta y)_{ts}\right|}{(t-s)^\beta} \\
&\leq \sup\limits_{0\leq s< t \leq T} \frac{\left|R^Y_{ts}\right|}{(t-s)^\beta}
		 + \sup\limits_{0\leq s< t \leq T} \frac{\left|(S(t-s)-\mbox{Id}) y_s\right|}{(t-s)^\beta}
		 + \sup\limits_{0\leq s< t \leq T} \frac{\left|\omega^S_{ts}(G(y_s))\right|}{(t-s)^\beta} \\
&\leq \left\|R^Y \right\|_{2\beta,T} T^\beta 
		+ C \left\|y \right\|_{\infty,D_{2\beta},T} T^{\beta} 
		+ C \ltn \omega\rtn_{\alpha,T} T^{\alpha-\beta} \\
&\leq C \left(T^{\alpha-\beta} + T^\beta \Phi_T(y,z) \right),
\end{align*}
which proves the first statement. \\

Furthermore, due to~\eqref{estimate_b}, the estimates for $z$ result in
\begin{align*}
\left\|z \right\|_{\alpha+\beta,T}
&= \sup\limits_{0\leq s< t \leq T}\sup\limits_{\left| E \right|\leq 1} \frac{\left|z_{ts}(E) \right|}{(t-s)^{\alpha+\beta}} \\
&\leq \sup\limits_{0\leq s< t \leq T}\sup\limits_{\left| E \right|\leq 1} \frac{\left|R^Z_{ts}(E) \right|}{(t-s)^{\alpha+\beta}}
		 + \sup\limits_{0\leq s< t \leq T}\sup\limits_{\left| E \right|\leq 1} \frac{\left|b_{ts}(E,G(y_s)) \right|}{(t-s)^{\alpha+\beta}} \\
&\leq \left\|R^Z \right\|_{\alpha+2\beta,T} T^\beta 
		+ C\left(\ltn \omega \rtn_{\alpha}^2 + \left\| \omegaa \right\|_{2\alpha} \right) T^{\alpha-\beta} \\
&\leq C \left(T^{\alpha-\beta} + T^\beta \Phi_T(y,z) \right).
\end{align*}
\qed 
\end{proof}\\
The next result indicates the connection between the space-regularity of $y$ and of the initial \\data $\xi$. 
\begin{lemma}\label{lemma_y_2beta}
	Let $\xi \in D_{\beta}$ and $(y,z)$ be a fixed-point of $\aM_{T,\omega,\xi}$. Then for $t\in(0,T]$ we have that $y_{t}\in D_{2\beta}$.
\end{lemma}
\begin{proof}
By Theorem~\ref{theorem_localsolution} we know that $y \in C^\beta$ and $z \in C^{\alpha+\beta}$. In order to apply Corollary~\ref{corollary_sewing_epsestimate_adapted} we have to
compute
	\begin{align*}
(\hat\delta_{2}\Xi^{(y)})_{vmu}=\Xi^{(y)}_{vu} - \Xi^{(y)}_{vm} - S(v-m) \Xi^{(y)}_{mu}.
\end{align*}

Due to~\eqref{algebraic_omegaS} we immediately obtain that
\begin{align}\label{deltay}
(\hat\delta_2\Xi^{(y)})_{vmu}
&= \omega^S_{vm}(G(y_u)-G(y_m)) + (\hat\delta_2 z)_{vmu}(DG(y_u)) + z_{vm}(DG(y_u)-DG(y_m)).
\end{align}
Therefore, we estimate
\begin{align}\label{delta2dach}
\left|(\hat{\delta}_2 \Xi^{(y)})_{vmu} \right|
&\leq  \left|\omega^S_{vm}(G(y_u)- G(y_m) + DG(y_u) (\delta y)_{mu}) \right|
+  \left| z_{vm}(DG(y_u)-DG(y_m))\right|.
\end{align}
For the first term we have applying \eqref{estimate_omegaS} that 
\begin{align*}
\left|\omega^S_{vm}(G(y_u)- G(y_m) + DG(y_u) (\delta y)_{mu}) \right|
\leq C \ltn \omega\rtn_{\alpha} (v-m)^{\alpha} \left|G(y_u) - G(y_m) + DG(y_u)(\delta y)_{mu} \right|.
\end{align*}
Furthermore,
\begin{align*}
\left|G(y_u) - G(y_m) + DG(y_u)(\delta y)_{mu} \right|
&\leq C \ltn y\rtn_{\beta,T}^2 \left(m-u \right)^{2\beta}, 
\end{align*}
and
\begin{align*}
\left| z_{vm}(DG(y_u)-DG(y_m))\right| 
\leq C \left\|z \right\|_{\alpha+\beta} \ltn y\rtn_{\beta,T} 
\left(v-u \right)^{\alpha+2\beta}.
\end{align*}
Summarizing, we obtain
\begin{align}\label{est:delta2dach}
\left|(\hat{\delta}_2 \Xi^{(y)})_{vmu} \right|
\leq C \left(1+\ltn y\rtn_{\beta,T}^2 + \left\|z\right\|_{\alpha+\beta,T}^2 \right) \left(v-u \right)^{\alpha+2\beta}.
\end{align}
On the other hand
\begin{align*}
\left|S(v-u) \Xi^{(y)}_{vu} \right|_{D_{2\beta}}
&\leq \left|S(v-u) \omega^S_{vu}(G(y_u))\right|_{D_{2\beta}} + \left|S(v-u) z_{vu}(DG(y_u))\right|_{D_{2\beta}} \\
&\leq C \left(v-u \right)^{-\beta} \left| \omega^S_{vu}(G(y_u))\right|_{D_\beta}
+ C (v-u)^{-2\beta} \left|z_{vu}(DG(y_u)) \right|.
\end{align*}
Using \eqref{estimate_omegaS2} we get
\begin{align*}
\left|S(v-u) \Xi^{(y)}_{vu} \right|_{D_{2\beta}}
&\leq C (1+\left\|z\right\|_{\alpha+\beta,T}) \left(v-u \right)^{\alpha-\beta}. 
\end{align*}
Hence, Corollary \ref{corollary_sewing_epsestimate_adapted} entails
\begin{align*}
\left|(\hat\delta \aI \Xi^{(y)})_{ts} \right|_{D_{2\beta}} 
= \left|(\hat\delta y)_{ts} \right|_{D_{2\beta}}
\leq C \left(1+\ltn y\rtn_{\beta,T}^2 + \left\|z\right\|_{\alpha+\beta,T}^2 \right)(t-s)^{\alpha-\beta},
\end{align*}
which simply yields
\begin{align}
\left|y_t \right|_{D_{2\beta}}
&\leq \left|S(t)\xi \right|_{D_{2\beta}} + \left|(\hat\delta y)_{t0} \right|_{D_{2\beta}} \nonumber\\
&\leq C t^{-\beta} \left| \xi \right|_{D_\beta}
	+ C \left(1+\ltn y\rtn_{\beta,T}^2 + \left\|z\right\|_{\alpha+\beta,T}^2 \right) t^{\alpha-\beta}.\label{y:d2beta}
\end{align}
This proves the statement. 
\qed 
\end{proof}

\begin{lemma}\label{phi:endlich}
	If $\xi\in D_{2\beta}$ and $(y,z)$ is a fixed-point of $\mathcal{M}_{T,\omega,\xi}$ then $\Phi_{T}(y,z)<\infty$.
\end{lemma}
\begin{proof}
	First of all, note that if $\xi\in D_{2\beta}$,~\eqref{y:d2beta} immediately entails that $\|y\|_{\infty,2\beta,T}<\infty$.\\
	We now investigate $R^{Y}$. To this aim, we verify~\eqref{assumption_Xi} using~\eqref{estimate_omegaS}. This obviously results in
	\begin{align*}
	\left| \Xi^{(y)}_{vu} \right|
	&\leq \left| \omega^S_{vu}(G(y_u))\right| + \left| z_{vu}(DG(y_u))\right| \\
	&\leq C \|\omega\|_{\alpha}  (v-u)^{\alpha}	+ C \|z\|_{\alpha}(v-u)^{\alpha}.
	\end{align*}
Now~\eqref{assumption_deltaXi1} is verified by~\eqref{est:delta2dach}. Therefore we obtain~\eqref{property_IXiest1}, namely
\begin{align*}
|(\hat\delta\aI\Xi^{(y)})_{ts}-\Xi^{(y)}_{ts} |\leq C \left(1+\ltn y\rtn_{\beta,T}^2 + \left\|z\right\|_{\alpha+\beta,T}^2 \right) (t-s)^{\alpha+\beta}.
\end{align*}
This yields 
\begin{align*}
|R^{Y}_{ts}|&= |(\hat\delta y)_{ts} - \omega^{S}_{ts}(G(y_{s}))|
\leq |(\hat\delta\aI\Xi^{(y)})_{ts}-\Xi^{(y)}_{ts} | + |z_{ts}(DG(y_{s}))|\\ 
&\leq C \left(1+\ltn y\rtn_{\beta,T}^2 + \left\|z\right\|_{\alpha+\beta,T}^2 + \left\|z\right\|_{\alpha+\beta,T} \right) (t-s)^{\alpha+\beta}.
\end{align*}
Clearly, we infer from the previous computation that $\|R^{Y}\|_{2\beta,T}<\infty$, regarding also Theorem~\ref{theorem_localsolution}.\\

We now prove that $\|R^{Z}\|_{\alpha+2\beta,T}<\infty$. We make the same deliberations as for $R^Y$. \\   
Estimates \eqref{estimate_a} and \eqref{estimate_b} entail
\begin{align*}
|\Xi^{(z)}_{vu}(E)| \leq |b_{vu}(E,G(y_u))| + |a_{vu}(E,y_u)| 
\leq C |E| (1+\left\|y \right\|_{\infty,T}) (v-u)^{\alpha}
\end{align*}
and
\begin{align*}
\left|(\hat\delta_2 \Xi^{(z)})_{vmu}(E) \right|
&= |a_{vm}(E, \omega^S_{mu}(G(y_u))-(\hat\delta y)_{mu}) + b_{vm}(E,G(y_u)-G(y_m))| \\
&\leq C \ltn \omega \rtn_{\alpha} (v-m)^{\alpha} \left| E \right| \left\|R^Y \right\|_{2\beta,T} (m-u)^{2\beta} \\
&\quad+ C \left(\ltn \omega\rtn_{\alpha}^2 + \left\|\omegaa \right\|_{2\alpha} \right) (v-m)^{2\alpha} \left| E \right| \ltn y \rtn_{\beta,T} (m-u)^{\beta} \\
&\leq C |E| (\left\|R^Y \right\|_{2\beta,T} + \ltn y \rtn_{\beta,T})    (v-u)^{\alpha+2\beta}.
\end{align*}
Thus \eqref{property_IXiest1} implies 
\begin{align}\label{estimate_xi_z}
\left|(\hat\delta \aI \Xi^{(z)})_{ts}(E) - \Xi^{(z)}_{ts}(E) \right|
\leq C |E| (\left\|R^Y \right\|_{2\beta,T} + \ltn y \rtn_{\beta,T}) (t-s)^{\alpha+2\beta}.
\end{align}
Consequently
\begin{align*}
|R^{Z}_{ts}(E)| &= |z_{ts}(E) - b_{ts}(E,G(y_s))| \\
&\leq \left|(\hat\delta \aI \Xi^{(z)})_{ts}(E) - \Xi^{(z)}_{ts}(E) \right| + |a_{ts}(E,y_s) - \omega^S(E y_s)|. 
\end{align*}
Applying \eqref{estimate_a2} yields
\begin{align*}
|R^{Z}_{ts}(E)|
&\leq C |E| (\left\|R^Y \right\|_{2\beta,T} + \ltn y \rtn_{\beta,T}) (t-s)^{\alpha+2\beta} 
 + C |E| |y_s|_{D_{2\beta}} (t-s)^{\alpha + 2\beta}.
\end{align*}
This proves that $\|R^{Z}\|_{\alpha+2\beta,T} < \infty$.  
	\qed
	\end{proof}\\

 We now derive the following a-priori estimate of the solution mapping of~\eqref{eq1}. The computations rely on similar arguments as in the previous Lemma.
\begin{lemma}\label{lemma_estimate}
Let $\xi \in D_{2\beta}$ and let $(y,z)$ be a fixed-point of $\aM_{T,\omega,\xi}$ with $0 < T\leq 1$. 
Then it holds
\begin{align}
\Phi_T(y,z) \leq C \left(\left|\xi \right|_{D_{2\beta}} + T^{\alpha-\beta} + T^{\alpha}  \Phi_T(y,z) \right).
\end{align}
\end{lemma}
\begin{proof}
Recall that $\Phi_T(y,z)= \left\|y \right\|_{\infty,D_{2\beta},T} + \left\|R^Y \right\|_{2\beta,T} + \left\|R^Z \right\|_{\alpha+2\beta,T}$. We begin with $\left\|R^Y \right\|_{2\beta,T}$.  
and further use that
\begin{align*}
\left|G(y_u) - G(y_m) + DG(y_u)(\delta y)_{mu} \right|
&= \Bigg|\int\limits_{0}^{1} \left[DG(y_u + q (\delta y)_{mu}) - DG(y_u)\right]  dq~ (\delta y)_{mu} \Bigg| \\
&\leq \int\limits_{0}^{1} \left|DG(y_u + q (\delta y)_{mu}) - DG(y_u)\right| dq \\
&\quad\quad\cdot \left[\left|R^Y_{mu}\right| + \left|(S(m-u)-\mbox{Id})y_u \right| + \left|\omega^S_{mu}(G(y_u)) \right|\right] \\[0.6ex]
&\leq C \left[\left\|R^Y \right\|_{2\beta,T} (m-u)^{2\beta} + \left\| y\right\|_{\infty,D_{2\beta},T} (m-u)^{2\beta}\right. \\
		 &\quad\quad \left.+ \ltn y\rtn_{\beta,T} \ltn \omega \rtn_{\alpha} (m-u)^{\alpha+\beta} \right].
\end{align*}
Applying \eqref{estimate_y_AT} results in
\begin{align*}
\left|G(y_u) - G(y_m) + DG(y_u)(\delta y)_{mu} \right|
\leq C \left[\Phi_T(y,z)(m-u)^{2\beta} + (T^{\alpha-\beta} + T^\beta \Phi_T(y,z)) (m-u)^{\alpha+\beta}\right].
\end{align*} 
All in all we obtain for the first term in~\eqref{delta2dach}
\begin{align*}
\left|\omega^S_{vm}(G(y_u)- G(y_m) + DG(y_u) (\delta y)_{mu}) \right|
\leq C (1+ \Phi_T(y,z)) (v-u)^{\alpha+2\beta}.
\end{align*}
For the second term in~\eqref{delta2dach} we have 
\begin{align*}
&\left|z_{vm}(DG(y_u)-DG(y_m)) \right|\\
&\leq  \left|R^Z_{vm}(DG(y_u) -DG(y_m)) \right| 
		 + \left|b_{vm}(DG(y_u)-DG(y_m),G(y_m)) \right|\\
&\leq C \left\|R^Z \right\|_{\alpha+2\beta,T} (v-m)^{\alpha+2\beta}
		 + C \left(\ltn \omega \rtn_{\alpha}^2 + \left\|\omegaa \right\|_{2\alpha} \right) (v-m)^{2\alpha} \ltn y\rtn_{\beta,T} (m-u)^{\beta}.
\end{align*}
Again, we apply \eqref{estimate_y_AT} and derive
\begin{align*}
\left|z_{vm}(DG(y_u)-DG(y_m)) \right|
&\leq C (1+\Phi_T(y,z)) (v-u)^{\alpha+2\beta}.
\end{align*}
Summarizing, we obtain
\begin{align*}
\left|(\hat{\delta}_2 \Xi^{(y)})_{vmu} \right| \leq C(1+\Phi_T(y,z)) \left(v-u \right)^{\alpha+2\beta}.
\end{align*}
Then~\eqref{property_IXiest1} yields
\begin{align*}
\left|(\hat\delta \aI \Xi^{(y)})_{ts} - \Xi^{(y)}_{ts} \right|
\leq C (1+\Phi_T(y,z)) \left(t-s \right)^{\alpha+2\beta}.
\end{align*}
Consequently,
\begin{align*}
\left| R^Y_{ts} \right| 
&= \left| (\hat\delta y)_{ts} - \omega^S_{ts}(G(y_s)) \right| \\
&\leq \left|(\hat\delta \aI \Xi^{(y)}_{ts} - \Xi^{(y)}_{ts} \right| + \left|z_{ts}(DG(y_s)) \right|\\
&\leq C (1+\Phi_T(y,z)) \left(t-s \right)^{\alpha+2\beta} + \left\|z \right\|_{\alpha+\beta,T} \left(t-s \right)^{\alpha+\beta}.
\end{align*}
Now \eqref{estimate_z_AT}  entails the first important estimate on the $2\beta$-norm of $R^{Y}$, namely
\begin{align}
\left\|R^Y \right\|_{2\beta,T} \leq C \left(T^{\alpha-\beta} + T^\alpha \Phi_T(y,z) \right) \label{estimate_RY}.
\end{align}
We now continue investigating $\left\|y\right\|_{\infty,D_{2\beta},T}$. In order to apply Corollary~\ref{corollary_sewing_epsestimate_adapted} we firstly consider
\begin{align*}
\left|S(v-u) \Xi^{(y)}_{vu} \right|_{D_{2\beta}}
&\leq \left|S(v-u) \omega^S_{vu}(G(y_u))\right|_{D_{2\beta}} + \left|S(v-u) z_{vu}(DG(y_u))\right|_{D_{2\beta}} \\
&\leq C \left(v-u \right)^{-\beta} \left| \omega^S_{vu}(G(y_u))\right|_{D_\beta}
		 + C (v-u)^{-2\beta} \left|z_{vu}(DG(y_u)) \right|.
\end{align*}
Using~\eqref{estimate_omegaS2} and~\eqref{estimate_z_AT} we get
\begin{align*}
\left|S(v-u) \Xi^{(y)}_{vu} \right|_{D_{2\beta}}
&\leq C (1+\left\|z\right\|_{\alpha+\beta,T}) \left(v-u \right)^{\alpha-\beta} \\
&\leq C(1+T^\beta \Phi_T(y,z)) (v-u)^{\alpha-\beta}.
\end{align*}
Hence, by Corollary~\ref{corollary_sewing_epsestimate_adapted} we obtain
\begin{align*}
\left|(\hat\delta \aI \Xi^{(y)})_{ts} \right|_{D_{2\beta}} 
= \left|(\hat\delta y)_{ts} \right|_{D_{2\beta}} 
&\leq C(1+T^\beta \Phi_T(y,z)) (t-s)^{\alpha-\beta} + C(1+\Phi_T(y,z))  (t-s)^{\alpha} \\ 
&\leq C(T^{\alpha-\beta}+T^\alpha \Phi_T(y,z)).
\end{align*}
Regarding this we immediately obtain
\begin{align*}
\left|y_t \right|_{D_{2\beta}} 
&\leq \left|(\hat\delta \aI \Xi^{(y)})_{t0} \right|_{D_{2\beta}} + \left|S(t)\xi\right|_{D_{2\beta}} \\
&\leq C(\left|\xi \right|_{D_{2\beta}}+ T^{\alpha-\beta} + T^\alpha \Phi_T(y,z)). 
\end{align*}
This obviously implies the second important estimate, namely
\begin{align}
\left\| y \right\|_{\infty,D_{2\beta},T}
&\leq C(\left|\xi \right|_{D_{2\beta}}+ T^{\alpha-\beta} + T^\alpha \Phi_T(y,z)). \label{estimate_y_infty_D2beta}
\end{align}
Finally, we only have to compute $\left\|R^Z \right\|_{\alpha+2\beta,T}$ analogously to the proof of Lemma~\ref{phi:endlich}. \\

Applying~\eqref{estimate_RY} and~\eqref{estimate_y_AT} to~\eqref{estimate_xi_z} entails
\begin{align*}
\left|(\hat\delta \aI \Xi^{(z)})_{ts}(E) - \Xi^{(z)}_{ts}(E) \right| 
\leq C(T^{\alpha-\beta} + T^\alpha \Phi_T(y,z) ) \left| E \right| \left(t-s \right)^{\alpha+2\beta}.
\end{align*}
Consequently,~\eqref{estimate_a2} further leads to
\begin{align*}
\left|R^Z_{ts}(E) \right| 
&\leq \left|(\hat\delta \aI \Xi^{(z)})_{ts}(E) - \Xi^{(z)}_{ts}(E) \right| 
		 + \left|a_{ts}(E,y_s)-  \omega^S_{ts}(E y_s) \right| \\
&\leq C(T^{\alpha-\beta} + T^\alpha \Phi_T(y,z) ) \left| E \right| \left(t-s \right)^{\alpha+2\beta}
		 + C \ltn \omega \rtn_{\alpha} \left\|y \right\|_{\infty,D_{2\beta,T}} \left| E \right| \left(t-s \right)^{\alpha+2\beta}.
\end{align*}
Regarding this and plugging in~\eqref{estimate_y_infty_D2beta}, we derive the third and final important estimate for the terms defining $\Phi_{T}$, namely
\begin{align}
\left\|R^Z \right\|_{\alpha+2\beta,T} 
\leq C \left (\left| \xi \right|_{D_{2\beta}} + T^{\alpha-\beta} + T^\alpha \Phi_T(y,z) \right) \label{estimate_RZ}.
\end{align}
This proves the statement, i.e.
\begin{align*}
\Phi_T(y,z) &\leq C\left(\left| \xi \right|_{D_{2\beta}} + T^{\alpha-\beta} + T^\alpha \Phi_T(y,z) \right).
\end{align*}
\qed 
\end{proof}\\
We now derive a crucial estimate which will be required for the concatenation procedure.
\begin{lemma}\label{lemma_apriori_estimate}
Let $T>0$, $r \geq 1 \vee \left|\xi\right|_{D_{2\beta}}$ and let $(y,z)$ be a fixed-point of $\aM_{T,\omega,\xi}$. Then there exists a constant $M>0$ independent of $r$, such that
\begin{align*}
\left\|y \right\|_{\infty,D_{2\beta,T}} \leq r M e^{MT}.
\end{align*} 
\end{lemma}
\begin{proof}
By Remark \ref{remark_localsolution_cut} we know that by restricting the solution on a smaller time interval $[0,\tilde{T}]$, with $\tilde{T}<T$, we obtain a fixed-point of $\aM_{\tilde{T},\omega,\xi}$. 
According to Lemma \ref{lemma_estimate} we have for all $0<\tilde{T}\leq 1$ that
\begin{align*}
\Phi_{\tilde{T}}(y,z) \leq C \left(\left|\xi \right|_{D_{2\beta}} + \tilde{T}^{\alpha-\beta} + \tilde{T}^{\alpha} \Phi_{\tilde{T}}(y,z) \right).
\end{align*}
We now choose $0<T_0^*\leq \tilde{T}$ sufficiently small such that $C (T_0^*)^\alpha \leq \frac{1}{2}$.
This yields for all $T_0<T_0^*$ that
\begin{align*}
\Phi_{T_{0}}(y,z) \leq 2C\left(\left|\xi \right|_{D_{2\beta}}+1 \right) \leq 4C r.
\end{align*}
Consequently, this means that
\begin{align*}
 \left\|y \right\|_{\infty,D_{2\beta}, T_0} \leq 4C r.
\end{align*}
At this point it is important to note that the choice of $T_0^*$ is independent of $r$ and $T$. \\

If $T\leq T_0^*$ the statement follows choosing $M\geq 4C$. Otherwise we can find an $N \in \IN $ (not necessarily unique), such that $\frac{T_0^*}{2}< \frac{T}{N}\leq T_0^*$. In this case we set $T_0:=\frac{T}{N}$.\\

Now, combining Lemma~\ref{lemma_estimate} and Lemma~\ref{lemma_solution_shift} we obtain for $1 \leq n \leq N-1$ that
\begin{align*}
\Phi_{T_0}(\widetilde\theta_{nT_0}y,\widetilde\theta_{nT_0}z) 
&\leq C \left(\left|y_{n T_0} \right|_{D_{2\beta}} + T_0^{\alpha-\beta} + T_0^{\alpha} \Phi_{T_0}(\widetilde\theta_{nT_0}y,\widetilde\theta_{nT_0}z) \right),
\end{align*}
Since $CT^{\alpha}_{0}\leq \frac{1}{2}$ and $y_{nT_{0}} =(\widetilde{\theta}_{(n-1)T_{0}}y)_{T_{0}}$,
the previous estimate results in
\begin{align*}
\Phi_{T_0}(\widetilde\theta_{nT_0}y,\widetilde\theta_{nT_0}z)
\leq 2C \Big(\Big\|\widetilde\theta_{(n-1)T_{0}}y \Big\|_{\infty,D_{2\beta},T_0} + 1 \Big),
\end{align*}
which yields
\begin{align*}
\Big\| \widetilde\theta_{nT_0}y \Big\|_{\infty,D_{2\beta},T_0}
&\leq \Phi_{T_0}(\widetilde\theta_{nT_0}y,\widetilde\theta_{nT_0}z)
\leq 2C \Big(\Big\|\widetilde\theta_{(n-1)T_{0}}y \Big\|_{\infty,D_{2\beta},T_0} + 1 \Big).
\end{align*}
By induction we infer that
\begin{align*}
\left\| \widetilde\theta_{nT_0}y \right\|_{\infty,D_{2\beta},T_0} 
&\leq \left(4C \right)^{n+1} r, \quad \text{for all } n = 0,\ldots,N-1.
\end{align*}
From this we finally conclude
\begin{align*}
\left\|y \right\|_{\infty,D_{2\beta},T}
= \max\limits_{n=0,\ldots,N-1} \left\| \widetilde\theta_{nT_0}y \right\|_{\infty,D_{2\beta},T_0}
\leq \left(4C \right)^{N} r 
= \left(4C \right)^{\frac{T}{T_0}} r
\leq \left(\left(4C \right)^{\frac{2}{T_0^*}}\right)^{T} r \leq M e^{MT} r.
\end{align*}
for a sufficiently large $M$.
\qed 
\end{proof}

Now we state the main step required in order to obtain a global solution. We show that by the concatenation of two local solutions we obtain a solution on a larger time interval. For similar arguments and techniques, see \cite{GarridoLuSchmalfuss2}.
\begin{lemma}\label{lemma_concatenation}
Let $(y^1,z^1)$ be a fixed-point of $\aM_{T_1,\omega,\xi}$ and $(y^2,z^2)$ be a fixed-point of $\aM_{T_2,\theta_{T_1}\omega,y^1_{T_1}}$. Then we obtain a fixed-point $(y,z)$ of $\aM_{T_1+T_2,\omega,\xi}$ via
\begin{align*}
y_t:= 
\begin{cases}
y^1_t, & 0 \leq t \leq T_1 \\
y^2_{t-T_1}, & T_1 \leq t \leq T_1+T_2,
\end{cases}
\end{align*}
and
\begin{align*}
z_{ts}(E) =
\begin{cases}
z^1_{ts}(E), & 0 \leq s \leq t \leq T_1 \\
\omega^S_{tT_1}(E (\delta y^1)_{T_1 s}) + z^2_{t-T_1,0}(E) + S(t-T_1) z^1_{T_1 s}(E), & 0 \leq s \leq T_1 \leq t \leq T_1+T_2 \\
z^2_{t-T_1,s-T_1}(E), & T_1 \leq s \leq t \leq T_1+T_2.
\end{cases}
\end{align*}
\end{lemma}
\begin{proof}
The statement follows by a standard computation. We only focus on certain cases, since the rest are straightforward.~For the beginning we consider $T_1 \leq t \leq T_1+T_2$. We recall that we use the notation $\Xi^{(y/z)}_{\omega}$ and $\Xi^{(y/z)}_{\theta_{\cdot}\omega}$ in order to indicate the appropriate shifts with respect to $\omega$.
\begin{align*}
S(t) \xi + \aI \Xi^{(y)}_{\omega}(y,z)_{t} 
&= S(t-T_1) S(T_1)\xi + S(t-T_1) \aI \Xi^{(y)}_{\omega}(y,z)_{T_1} + (\hat\delta\aI \Xi^{(y)}_{\omega}(y,z))_{tT_1} \\
&= S(t-T_1) \left(S(T_1) \xi + \aI \Xi^{(y)}_{\omega}(y^1,z^1)_{T_1} \right)
		+ (\hat\delta\aI \Xi^{(y)}_{\omega}(y^2_{\cdot-T_1},z_{\cdot-T_1,\cdot-T_1}))_{tT_1} \\
&= S(t-T_1) y^1_{T_1} + (\hat\delta\aI \Xi^{(y)}_{\omega}(y^2_{\cdot-T_1},z^2_{\cdot-T_1,\cdot-T_1}))_{tT_1}.
\end{align*}
Recall that 
\begin{align*}
\Xi^{(y)}_{\omega}(y^2_{\cdot-T_1},z^2_{\cdot-T_1,\cdot-T_1})_{vu} 
&= \omega^S_{vu}(G(y^2_{u-T_1})) + z^2_{v-T_1,u-T_1}(DG(y^2_{u-T_1})),
\end{align*}
which further leads to
\begin{align*}
 \widetilde{\theta}_{T_1} \Xi^{(y)}_{\omega}(y^2_{\cdot-T_1},z^2_{\cdot-T_1,\cdot-T_1}){vu}
&= \omega^S_{v+T_1,u+T_1}(G(y^2_u)) + z^2_{vu}(DG(y^2_u)) \\
&= \widetilde\theta_{T_1} \omega^S_{vu}(G(y^2_u)) + z^2_{vu}(DG(y^2_u)).
\end{align*}
Now, Lemma~\ref{lemma_sewing_shift} entails 
\begin{align*}
(\hat\delta\aI \Xi^{(y)}_{\omega}(y^2_{\cdot-T_1},z^2_{\cdot-T_1,\cdot-T_1}))_{tT_1} 
= (\hat\delta\aI \Xi^{(y)}_{\theta_{T_1} \omega}(y^2,z^2))_{t-T_1,0}.
\end{align*}
Consequently,
\begin{align*}
S(t) \xi + \aI \Xi^{(y)}_{\omega}(y,z)_{t} = S(t-T_1) y^1_{T_1} + \aI \Xi^{(y)}_{\theta_{T_1} \omega}(y^2,z^2)_{t-T_1} = y^2_{t-T_1} = y_t.
\end{align*}
Now, let $0 \leq s \leq T_1 \leq t \leq T_1+T_2$. Then we have
\begin{align*}
&(\hat\delta\aI \Xi^{(z)}_{\omega}(y,y))_{ts}(E) - \omega^S_{ts}(Ey_s) \\
= &(\hat\delta\aI \Xi^{(z)}_{\omega}(y,y))_{t T_1}(E) + S(t-T_1) (\hat\delta\aI \Xi^{(z)}_{\omega}(y,y))_{T_1 s}(E)
		 - \omega^S_{tT_1}(Ey_s) - S(t-T_1)  \omega^S_{T_1s}(Ey_s) \\
= &(\hat\delta\aI \Xi^{(z)}_{\omega}(y^2_{\cdot-T_1},y^2_{\cdot-T_1}))_{t T_1}(E) 
		+ S(t-T_1) (\hat\delta\aI \Xi^{(z)}_{\omega}(y^1,y^1))_{T_1 s}(E)
		 - \omega^S_{tT_1}(Ey^1_s) - S(t-T_1) \omega^S_{T_1s}(Ey^1_s) \\
= & S(t-T_1) z^1_{T_1 s}(E)  +\omega^S_{t T_1}(E(\delta y^1)_{T_1 s}) 
		+ (\hat\delta\aI \Xi^{(z)}_{\omega}(y^2_{\cdot-T_1},y^2_{\cdot-T_1}))_{t T_1}(E)
		- \omega^S_{t T_1}(E y^2_0),
\end{align*}
where we use in the last step that $y^2_0=y^1_{T_1}$. 
Since
\begin{align*}
\Xi^{(z)}_{\omega}(y^2_{\cdot-T_1},y^2_{\cdot-T_1})_{vu} (E)
&= b_{vu}(E,G(y^2_{u-T_1})) + a_{vu}(E,y^2_{u-T_1}) ,
\end{align*}
further entails
\begin{align*}
\widetilde{\theta}_{T_1} \Xi^{(z)}_{\omega}(y^2_{\cdot-T_1},y^2_{\cdot-T_1})_{vu} (E)
&= \widetilde\theta_{T_1} b_{vu}(E,G(y^2_{u})) + \widetilde\theta_{T_1} a_{vu}(E,y^2_{u}) \\
&= \Xi^{(z)}_{\theta_{T_1}\omega}(y^2,y^2)_{vu} (E).
\end{align*}
Hence, we infer using Lemma \ref{lemma_sewing_shift} that
\begin{align*}
&(\hat\delta\aI \Xi^{(z)}_{\omega}(y,y))_{ts}(E) - \omega^S_{ts}(Ey_s) \\
= & S(t-T_1) z^1_{T_1 s}(E)  +\omega^S_{t T_1}(E(\delta y^1)_{T_1 s}) 
		+ (\hat\delta\aI \Xi^{(z)}_{\theta_{T_1}\omega}(y^2,y^2))_{t-T_1,0}(E) - \omega^S_{t T_1}(E y^2_0) \\
= & S(t-T_1) z^1_{T_1 s}(E)  +\omega^S_{t T_1}(E(\delta y^1)_{T_1 s}) + z^2_{t-T_1,0}(E)
= z_{ts}(E).
\end{align*}
\qed 
\end{proof} \\
Regarding all the previous deliberations we can now state the main results of this section.
\begin{theorem}\label{theorem_global}
Let $\xi$ in $D_{2\beta}$. Then for any $T>0$ there exists a unique global solution, i.e.~there exists a unique fixed-point of $\aM_{T,\omega,\xi}$.
\end{theorem}
\begin{proof}
Let $r=1 \vee \left|\xi \right|_{D_{2\beta}}$. 
By Lemma~\ref{lemma_apriori_estimate} we know that every fixed-point of $\aM_{T,\omega,\xi}$ must satisfy the estimate
\begin{align*}
\left\|y \right\|_{\infty,D_{2\beta,T}} 
\leq r M e^{MT}=:\widetilde{r}.
\end{align*}
Particularly, this means that $\left|y_t \right|_{D_{2\beta}}\leq \widetilde{r}$, for all $t \leq T$. 
Applying Theorem~\ref{theorem_localsolution} with $\left|\xi \right|_{D_{2\beta}}\leq \widetilde{r}$  entails the existence of a unique local solution on a time interval $[0,T^*]$, where $T^* = T^*(\widetilde{r})$, i.e. there is a unique fixed-point $(y,z)$ of $\aM_{T^*\!,\omega,\xi}$. For simplicity, since we can choose $T^*$ arbitrary small, we set  $N:=\frac{T}{T^*} \in \IN$ for $N \geq 2$.\\ 

Note that $\left|y_{T^*}\right| \leq \widetilde{r}$.
 Hence, we can derive by using again Theorem~\ref{theorem_localsolution} the existence of a unique fixed-point of $\aM_{T^*\!,\theta_{T^*}\omega,y_{T^*}}$.
Furthermore, Lemma~\ref{lemma_concatenation} shows that we can concatenate them and obtain a fixed-point $(y,z)$ of $\aM_{2T^*\!,\omega,\xi}$. Again we have $\left|y_{2T^*}\right| \leq \widetilde{r}$.\\

Iterating this argument entails the existence of a unique fixed-point $(y,z)$ of $\aM_{T,\omega,\xi}$ for any $T>0$. 
\qed 
\end{proof}
\begin{corollary}
Let $\xi$ in $W$. Then for any $T>0$ there exists a unique fixed-point of $\aM_{T,\omega,\xi}$.
\end{corollary}
\begin{proof}
Theorem \ref{theorem_localsolution} gives us the existence of a unique fixed-point $(y,z)$  of $\aM_{T_1,\omega,\xi}$, where $T_1=T_1(\omega,\xi)$. Furthermore due to~\eqref{estimate_y_beta} we obtain $y_{T_1}\in D_{\beta}$. Then by Theorem \ref{theorem_localsolution} we know that there exists a unique-fixed point of $\aM_{T_2,\theta_{T_1}\omega,y_{T_1}}$, which according to Lemma \ref{lemma_concatenation} can be concatenated with the previous one to a fixed-point $(y,z)$ of $\aM_{T_1+T_2,\omega,\xi}$. Lemma \ref{lemma_y_2beta} entails $y_{T_1+T_2} \in D_{2\beta}$. Hence, we are in the setting of Theorem \ref{theorem_global} and obtain the existence of a global fixed-point of $\aM_{T-T_1-T_2,\theta_{T_1+T_2}\omega,y_{T_1+T_2}}$.~Again, this 
can be concatenated to a fixed-point of $\aM_{T,\omega,\xi}$ due to Lemma \ref{lemma_concatenation}. This procedure gives us the global-in-time solution.
\qed 
\end{proof}
 \section{Random dynamical systems}\label{sect:rds}
Based on the results derived in the previous section we investigate random dynamical systems for~\eqref{eq1}. There are very few works that deal with random dynamical systems for SPDEs driven by nonlinear multiplicative rough noise, see for instance~\cite{Gess}.
In the finite-dimensional setting this topic was considered in~\cite{BailleulRiedelScheutzow}.\\

We start by introducing the next fundamental concept in the theory of random dynamical systems, which describes a model of the driving noise.

\begin{definition} 
	Let $(\Omega,\mathcal{F},\mathbb{P})$ stand for a probability space and 
	$\theta:\mathbb{R}\times\Omega\rightarrow\Omega$ be a family of 
	$\mathbb{P}$-preserving transformations (i.e.,~$\theta_{t}\mathbb{P}=
	\mathbb{P}$ for $t\in\mathbb{R}$) having following properties:
	\begin{description}
		\item[(i)] the mapping $(t,\omega)\mapsto\theta_{t}\omega$ is 
		$(\mathcal{B}(\mathbb{R})\otimes\mathcal{F},\mathcal{F})$-measurable;
		\item[(ii)] $\theta_{0}=\textnormal{Id}_{\Omega}$;
		\item[(iii)] $\theta_{t+s}=\theta_{t}\circ\theta_{s}$ for all 
		$t,s,\in\mathbb{R}$.
	\end{description}
	Then the quadrupel $(\Omega,\mathcal{F},\mathbb{P},(\theta_{t})_{t\in\mathbb{R}})$ 
	is called a metric dynamical system.
\end{definition}
Motivated by this we precisely describe the random input driving~\eqref{eq1}. Therefore, our aim is introduce the (canonical) probability space associated to a Hilbert space-valued $\alpha$-H\"older rough path. We recall that $\alpha\in(\frac{1}{3},\frac{1}{2})$ was fixed at the beginning of this work. An example is constituted by a trace-class $V$-valued fractional Brownian motion with Hurst index $H\in(1/3,1/2]$. In order to construct it, we recall that a \emph{two-sided} real-valued fractional Brownian motion $\widetilde{\beta}^{H}(\cdot)$ with a Hurst index $H\in(0,1)$ is a centered Gaussian process with covariance function
\begin{align*}
	\mathbb{E} (\widetilde{\beta}^{H}(t)\widetilde{\beta}^{H}(s)) =\frac{1}{2} (|t|^{2H} + |s|^{2H} - |t-s|^{2H}), ~~\mbox{for  } s,t\in\mathbb{R}.
\end{align*}
In order to introduce a $V$-valued process, we let $Q$ stand for a positive symmetric operator of \emph{trace-class} on $V$, i.e.~$\mbox{tr}_{V}Q<\infty$. This has a discrete spectrum which will be denoted by $(\lambda_{n})_{n\in\mathbb{N}}$. It is well-known that the eigenvectors $(e_{n})_{n\in\mathbb{N}}$ build an orthonormal basis in $V$. Then a $V$-valued two-sided $Q$-fractional Brownian motion $\omega(\cdot)$ is represented by
\begin{align}\label{q}
\omega(t) = \sum\limits_{n=1}^{\infty}\sqrt{\lambda_{n}}\widetilde{\beta}^{H}_{n}(t) e_{n}, ~~ t\in\mathbb{R}  ,
\end{align}
where $(\widetilde{\beta}^{H}_{n}(\cdot))_{n\in\mathbb{N}}$ is a sequence of one-dimensional independent standard two-sided fractional Brownian motions with the same Hurst parameter $H$ and $\mbox{tr}_{V}Q=\sum\limits_{n=1}^{\infty}\lambda_{n}<\infty$. In the following sequel we further fix $H\in(\frac{1}{3},\frac{1}{2}]$. \\

Keeping~\eqref{q} in mind we are justified to introduce the canonical probability space $(C_{0}(\mathbb{R},V), \mathcal{B}(C_{0}(\mathbb{R},V)), \mathbb{P}, \theta )$. Here $C_{0}(\mathbb{R},V)$ denotes the set of all $V$-valued continuous functions which are zero in zero endowed with the compact open topology and $\mathbb{P}$ is the fractional Gau\ss{}-measure which is uniquely determined by $Q$. As already introduced in Section~\ref{preliminaries}, we take for $\theta$ the usual Wiener-shift,
namely $$\theta_{\tau}\omega_{t}=\omega_{t+\tau}-\omega_{\tau},~~\mbox{for } \omega\in C_{0}(\mathbb{R},V).$$
It is well-known that the above introduced quadruple is a metric dynamical system. For our aims we restrict it to the set $\Omega:=C^{\alpha'}_0(\mathbb{R}, V)$ of all $\alpha'$-H\"older-continuous paths on any compact interval, where $\frac{1}{3}<\alpha<\alpha'<H\leq\frac{1}{2}$. We equip this set with the trace $\sigma$-algebra $\mathcal{F}:=\Omega\cap\mathcal{B}(C_{0}(\mathbb{R},V))$ and take the restriction of $\mathbb{P}$ as well. Then $\Omega\subset C_{0}(\mathbb{R},V)$ has full measure and is $\theta$-invariant. Moreover, the new quadrupel $(\Omega,\mathcal{F},\mathbb{P},\theta)$ as introduced above forms again a metric dynamical system which we will further be restricted later on.\\

We point out the following result regarding the existence/construction of the L\'evy-area $\omegaa$ for an element $\omega\in\Omega$. We stress the fact that it is necessary to let $\omega$ be $\alpha'$-H\"older continuous for $\frac{1}{3}<\alpha<\alpha'<H\leq\frac{1}{2}$. This is required in order to lift $\omega$ to a rough path $\bm{\omega}=(\omega,\omegaa)$. To this aim we furthemore have to consider the restriction of $\omega$ on compact intervals. The precise setting is stated below.
\begin{lemma}\label{levyarea} Let $\frac{1}{3}< \alpha < \alpha' < H\leq\frac{1}{2}$ and $\omega\in \Omega$ be a $Q$-fractional Brownian motion with Hurst index H.  Then there is a $\theta$-invariant subset $\Omega' \subset \Omega$ of full measure such that for any $\omega \in \Omega'$ and for any compact interval $J \subset \IR$ there exists a L\'evy-area $\omegaa \in C^{2\alpha}(\Delta_{J},V \otimes V)$ such that  $\bm{\omega}=(\omega,\omegaa)$ defines an $\alpha$-H\"older rough path. This can further be approximated by a sequence $\bm\omega^{n}:=((\omega^{n},\omegaan))_{n\in\mathbb{N}}$ in the corresponding $d_{\alpha,J}$-metric. Here $(\omega^{n})_{n\in\mathbb{N}}$ are piecewise dyadic linear functions and 
	\begin{align*}
	\omegaan_{ts}=\int\limits_{s}^{t} (\delta \omega^{n})_{rs} \otimes d\omega^{n}_r.
	\end{align*}
\end{lemma}
\begin{proof}
	Let $j,k\in\mathbb{N}$ and $T\in\mathbb{N}$ be such that $J\subseteq[-T,T]$. We introduce
		\begin{align}
	\omegaa_{ts}(j,k):=\int\limits_{s}^{t} (\widetilde{\beta}^{H}_{j}(r) - \widetilde{\beta}^{H}_{j}(s))~ d \widetilde{\beta}^{H}_{k}(r),~~\mbox{for } -T\leq s \leq t \leq T.
	\end{align}
	This process exists \emph{almost surely} according to Theorem~2~in~\cite{CoutinQian}, see also Corollary~10~in~\cite{FritzHairer}.
	Regarding~\eqref{q} we can represent the inifinite-dimensional L\'evy-area $\omegaa_{ts}$ component-wise as
	\begin{align}
	\omegaa_{ts}= \sum\limits_{j,k=1}^{\infty}\sqrt{\lambda_{j}}\sqrt{\lambda_{k}}~ \omegaa_{ts}(j,k)~e_{j}\otimes e_{k}.
	\end{align}
	This is well-defined \emph{almost surely} due to the fact that $\mbox{tr}_{V}Q<\infty$.~Moreover one has that $\omegaan\to \omegaa$ in $C^{2\alpha}(\Delta_{[-T,T]}, V\otimes V)$ \emph{almost surely}. The proof of these assertions relies on a standard Borel-Cantelli argument combined with the Garsia-Rodemich-Rumsey inequality and follows the lines of Lemma~2~in~\cite{GarridoLuSchmalfuss3}. Since $J\subseteq[-T,T]$, one clearly concludes that $\bm\omega^{n}$ converges to $\bm\omega$ with respect to the $d_{\alpha,J}$-metric. This immediately yields that $\Omega'$
	has full measure and is $\theta$-invariant.
	\qed
\end{proof}\\

From now on we work with the metric dynamical system $(\Omega',\mathcal{F}', \mathbb{P}',\theta)$ corresponding to $\Omega'$ constructed in Lemma~\ref{levyarea}. As above we set $\mathcal{F}':=\Omega' \cap \mathcal{F}$ and take $\mathbb{P}'$ as the restriction of $\mathbb{P}$.\\

For the sake of completeness we indicate the following result regarding the shift-property of an $\alpha$-H\"older rough path.
\begin{lemma}\label{shift:rp}
	For an $\alpha$-H\"older rough path $(\omega,\omegaa)$ 
	and $\tau\in\mathbb{R}$, the time-shift $(\theta_{\tau}\omega,\widetilde{\theta}_{\tau}\omegaa)$ 
	\begin{align*}
	& \theta_{\tau} \omega_t  =\omega_{t+\tau} - \omega_{\tau}\\
	& \widetilde{\theta}_{\tau}\omegaa_{ts}=\omegaa_{t+\tau,s+\tau}
	\end{align*}
	is again an $\alpha$-H\"older rough path.\end{lemma} 
\begin{proof}
	The time-regularity is straightforward and one can easily verify Chen's relation~\eqref{chen}. This reads as
	\begin{align}
	\widetilde{\theta}_{\tau} \omegaa_{ts} - \widetilde{\theta}_{\tau} \omegaa_{us} 
	-\widetilde{\theta}_{\tau}\omegaa_{tu} &= \omegaa_{t+\tau,s+\tau} 
	- \omegaa_{u+\tau,s+\tau} - \omegaa_{t+\tau,u+\tau} \nonumber \\
	& = \omega_{u+\tau,s+\tau} \otimes \omega_{t+\tau,u+\tau} \label{chena},\\
	&=(\omega_{u+\tau} -\omega_{\tau} - \omega_{s+\tau} +\omega_{\tau} ) \otimes (\omega_{t+\tau} 
	- \omega_{\tau} -\omega_{u+\tau} + \omega_{\tau} ) \nonumber \\
	& =(\delta\theta_{\tau}\omega)_{us} \otimes (\delta\theta_{\tau} \omega)_{tu}. \nonumber
	\end{align}
	where in~\ref{chena} we use Chen's relation~\eqref{chen}.
	\qed
\end{proof}
\begin{definition}\label{rds} 
	A random dynamical system on $W$ over a metric dynamical 
	system $(\Omega,\mathcal{F},\mathbb{P},(\theta_{t})_{t\in\mathbb{R}})$ 
	is a mapping $$\varphi:\mathbb{R}_{+}\times\Omega\times W\to W,
	\mbox{  } (t,\omega,x)\mapsto \varphi(t,\omega,x), $$
	which is $(\mathcal{B}(\mathbb{R}_{+})\times\mathcal{F}\times
	\mathcal{B}(W),\mathcal{B}(W))$-measurable and satisfies:
	\begin{description}
		\item[(i)] $\varphi(0,\omega,\cdot{})=\textnormal{Id}_{W}$ 
		for all $\omega\in\Omega$;
		\item[(ii)]$ \varphi(t+\tau,\omega,x)=
		\varphi(t,\theta_{\tau}\omega,\varphi(\tau,\omega,x)), 
		\mbox{ for all } x\in W, ~t,\tau\in\mathbb{R}_{+},~\omega\in\Omega.$
	\end{description}
\end{definition}
If one additionally assumes that
\begin{description}
	\item[(iii)] $\varphi(t,\omega,\cdot{}):W\to W$ is 
	continuous for all $t\in\mathbb{R}_{+}$ and all $\omega\in\Omega$,
\end{description}
then $\varphi$ is called a \emph{continuous random dynamical system}.\\

The second property in Definition~\ref{rds} is referred to as the \emph{cocycle property}. One can now expect that 
the solution operator of~\eqref{eq1} generates a random dynamical 
system. Indeed, working with a pathwise interpretation of the 
stochastic integral, \emph{no exceptional 
sets} can occur. \\

We can now state the main result of this work. Recall that $\Omega'$ was constructed in Lemma~\ref{levyarea}.
\begin{theorem} The solution operator of~\eqref{eq1} generates a random dynamical system $\varphi:\mathbb{R}_{+}\times\Omega'\times W\to W$ given by
	\begin{align}
	\varphi(t,\omega,\xi):= y_{t},
	\end{align}
	where $y$ is the first component of the fixed-point operator $\mathcal{M}_{t,\omega,\xi}$.
\end{theorem}
\begin{proof}
Due to Theorem~\ref{theorem_global} we know that we can define the solution $(y,z)$ of~\eqref{eq1} on any time-interval $[0,T]$ for $T>0$. The cocycle property was proved in Lemma~\ref{lemma_solution_shift}. The continuity of $\varphi$ with respect to time and initial condition is clear, we only have to show the measurablity. Therefore we consider a sequence of solutions $((y^{n}, z^{n}))_{n\in\mathbb{N}}$ corresponding to the smooth approximations $((\omega^{n},\omegaan))_{n\in\mathbb{N}}$, recall Lemma~\ref{levyarea}. Note that the mapping $\omega \mapsto (\omega^n,\omegaan)$ is measurable. Due to the fact that $\omega^n$ is smooth $y^{n}$ is a classical solution of~\eqref{eq1}. Hence the mapping
\begin{align*}
[0,T]\times\Omega'\times W\ni (t,\omega,\xi)\mapsto y^{n}_{t}\in W
\end{align*}
is $(\mathcal{B}([0,T])\otimes\mathcal{F}'\otimes\mathcal{B}(W), \mathcal{B}(W))$-measurable.
Regarding Lemma~\ref{lemma:cont:dependence1} one can immediately infer that the solution $(y,z)$ continuously depends on $(\omega^{n},\omegaan)$. This leads to 
\begin{align}
\lim\limits_{n\to\infty} y^{n}_{t} = y_{t},
\end{align}
which gives us the measurability of $y_{t}$ with respect to $\mathcal{F}'\otimes \mathcal{B}(W)$. Since $y$ is continuous with respect to $t$, we obtain by Lemma~3~in~\cite{CastaingValadier} the jointly measurability, i.e.~the $(\mathcal{B}([0,T])\otimes\mathcal{F}'\otimes\mathcal{B}(W), \mathcal{B}(W))$ meaurability of the mapping
\begin{align}\label{meas}
[0,T]\times\Omega'\times W\ni (t,\omega,\xi)\mapsto y_{t}\in W.
\end{align}
Since~\eqref{meas} holds true for any $T>0$, one obviously concludes that
$\varphi$ is $(\mathcal{B}(\mathbb{R}_{+})\otimes\mathcal{F}'\otimes\mathcal{B}(W), \mathcal{B}(W))$-measurable.
	\qed
	\end{proof}



\begin{thebibliography}{10}
	\bibitem{Arnold}
	L.~Arnold.
	\newblock{\em Random Dynamical Systems}. Springer, Berlin Heidelberg, Germany, 2003.
	
	\bibitem{BailleulRiedelScheutzow}
	I.~Bailleul, S.~Riedel, M.~Scheutzow.
	\newblock Random dynamical system, rough paths and rough flows.
	\newblock {\em J. Differential. Equat.} 262(12):5792--5823, 2017.
	\bibitem{CastaingValadier}
	C.~Castaing, M.~Valadier. 
	\newblock {\em Convex analysis and measurable multifunctions}.
	\newblock Springer-Verlag, Berlin, 1977. Lecture Notes in Mathematics, Vol. 580.
	\bibitem{CoutinLejay}
	L.~Coutin, A.~Lejay.
	\newblock{Sensitivity of rough differential equations:
		an approach through the Omega lemma}.
	\newblock{\em arXiv:1712.04705v1}, pages 1--, 2017.
	\bibitem{CoutinQian}
	L.~Coutin, Z.~Qian.
	\newblock Stochastic analysis, rough path analysis and fractional Brownian motions.
	\newblock {\em Probab. Theory. Relat. Fields}, 122(1):108--140, 2002.
	
	\bibitem{DeyaGubinelliTindel}
	A.~Deya, M.~Gubinelli, S.~Tindel.
	\newblock{Non-linear rough heat equations}.
	\newblock{\em Probab. Theory Related Fields}.
	\newblock 153(1--2):97--147, 2012.
	
	\bibitem{DejaN}
	A.~Deya, A.~Neuenkirch, S.~Tindel.
	\newblock A Milstein-type scheme without L\'evy area terms for
	SDES driven by fractional Brownian motion.
	\newblock{\em Ann. Inst. H. Poincar\'e Probab. Statist.}, 48(2),
	518--550, 2012.
	\bibitem{Gess}
	B.~Fehrman, B.~Gess.
	\newblock Well-posedness of stochastic porous media equations with nonlinear, conservative noise.
	\newblock {\em arXiv:1712.05775}, pages 1--, 2018.
	\bibitem{FritzHairer}
	P.K.~Friz, M.~Hairer.
	\newblock {\em A Course on Rough Paths}.
	\newblock Springer, 2014.
	\bibitem{FritzVictoir}
	P.K.~Friz, N.B.~Victoir.
	\newblock{\em Multidimensional Stochastic Processes as
		Rough Paths: Theory and Applications}.
	Cambridge Studies in Advanced Mathematics, 2010.
	\bibitem{Gao}
	H.~Gao, M.J.~Garrido-Atienza, B.~Schmalfu\ss{}.
	\newblock Random attractors for stochastic evolution
	equations driven by fractional Brownian motion.
	\newblock{\em SIAM J. Math. Anal.}, 46(4):2281--
	2309, 2014.
	
	\bibitem{GS}
	M.J.~Garrido-Atienza, B.~Schmalfu\ss{}.
	\newblock Local Stability of Differential Equations Driven by H\"older-Continuous Paths
	with H\"older Index in (1/3,1/2).
	\newblock {\em SIAM J. Appl. Dyn. Syst.}, 17(3):2352--2380, 2018.
		
	\bibitem{GLS}
	M.J.~Garrido-Atienza, K.~Lu, B.~Schmalfu\ss{}.
	\newblock Random dynamical systems for stochastic partial differential equations driven by fractional Brownian motion.
	\newblock {\em Discrete Contin. Dyn. Syst. Ser. B}, 14(2):473--493, 2010.
\bibitem{GarridoLuSchmalfussUnstable}
M.J.~Garrido-Atienza, K.~Lu, B.~Schmalfu\ss{}.
\newblock Unstable invariant manifolds for stochastic PDEs driven by a fractional Brownian motion.
\newblock {\em J. Differential Equat}. 248(7):1637--1667, 2010.
	
	\bibitem{GarridoLuSchmalfuss1}
	M.J.~Garrido-Atienza, K.~Lu, B.~Schmalfu\ss{}.
	\newblock{Local pathwise solutions to stochastic evolution equations driven by fractional Brownian motions with Hurst parameters $H\in(1/3,1/2]$}.
	\newblock{\em Discrete Cont. Dyn-B}. 20(8):2553--2581, 2015.
	
	\bibitem{GarridoLuSchmalfuss3}
	M.J.~Garrido-Atienza, K.~Lu, B.~Schmalfuss{}.
	\newblock{\em L\'evy-areas of Ornstein-Uhlenbeck processes in Hilbert-spaces}.
	\newblock{Continuous and Distributed Systems II}. pp, 167--188, 2015.
	
	\bibitem{GarridoLuSchmalfuss2}
	M.J.~Garrido-Atienza, K.~Lu, B.~Schmalfu\ss{}.
	\newblock{Random dynamical systems for stochastic evolution equations driven by multiplicative fractional Brownian noise with Hurst parametes $H\in(1/3,1/2]$}.
	\newblock{\em SIAM J. Appl. Dyn. Syst}. 15(1), 625--654, 2016.
	
	\bibitem{Gubinelli}
	M.~Gubinelli.
	\newblock Controlling rough paths.
	\newblock {\em J. Func. Anal}. 216(1):86--140, 2004.
	
	\bibitem{GubinelliLejayTindel}
	~M. Gubinelli, A.~Lejay, S.~Tindel.
	\newblock Young integrals and SPDEs.
	\newblock {\em Potential Anal}. 25(4):307--326, 2006.
	
	\bibitem{GubinelliTindel}
	M.~Gubinelli, S.~Tindel.
	\newblock Rough evolution equations.
	\newblock {\em Ann. Probab}. 38(1):1--75, 2010.
	\bibitem{Hairer}
	M.~Hairer.
	\newblock A theory of regularity structures. 
	\newblock {\em Invent. Math.}, 198(2):269--504, 2014.
	\bibitem{HiarerE}
	M.~Hairer.
	\newblock Ergodicity of stochastic differential equations driven by fractional Brownian motion.
	\newblock {\em Ann. Probab.}, 33(2):703--758, 2005.
	\bibitem{HesseNeamtu}
	R.~Hesse, A.~Neamtu
	\newblock{Local mild solutions for rough stochastic partial differential
		equations}.
	\newblock{\em arXiv:1809.08616}, pages 1--, 2018. 
	
	\bibitem{HuNualart}
	Y.~Hu, D.~Nualart.
	\newblock{Rough path analysis via fractional calculus}.
	\newblock{\em Trans. Amer. Math. Soc}. 361(5):2689--2718, 2009. 
	\bibitem{ImkellerLederer}
	P.~Imkeller, C.~Lederer. 
	\newblock The Cohomology of Stochastic and Random Differential Equations
	and local linearization of stochastic flows.
	\newblock {\em Stoch. Dyn.} 2(2):131--159, 2002.
	\bibitem{FA}
	L.W.~Kantorowitsch, G.P.~ Akilow.
	\newblock {\em Funktionalanalysis in normierten R\"aumen}. Verlag
	Harri Deutsch, 1978.
	\bibitem{Kunita}
	H.~Kunita.
	\newblock {\em Stochastic flows and stochastic differential equations}. \newblock Cambridge University Press, 1990.
	\bibitem{Lederer}
	C.~Lederer.
	\newblock {\em Konjugation stochastischer und zuf\"alliger station"arer Differentialgleichungen und
	eine Version des lokalen Satzes von Hartman-Grobman f\"ur stochastische Differentialgleichungen}.
	\newblock PhD Thesis. HU Berlin, 2001.
	
	\bibitem{Lunardi}
	A.~Lunardi.
	\newblock{\em Analytic semigroups and optimal regularity in parabolic problems}.
	\newblock Birkh\"auser, 1995.
	
	\bibitem{MaslowskiNualart}
	B.~Maslowski, D.~Nualart.
	\newblock Evolution equations driven by a fractional Brownian motion.
	\newblock {\em J. Funct. Anal.}, 202(1):277--305,2003.
	
	\bibitem{MaslowskiS}
	B.~Maslowski, B.~Schmalfu\ss{}.
	\newblock Random dynamical systems and stationary solutions of differential equations driven by the fractional Brownian motion.
	\newblock {\em Stochastic Anal. Appl.}, 22(6):1577--1607, 2004.
	\bibitem{Mohammed}
	S.~Mohammed, T.~Zhang, H.~Zhao.
	\newblock {\em The stable manifold theorem for semilinear stochastic evolution equations and stochastic partial differential equations}.
	\newblock Memoirs of the AIMS, vol.~196, nr.~197, 2008.
	\bibitem{NualartRascanu}
	D.~Nualart, A.~R\u a\c scanu.
	\newblock Differential equations driven by fractional Brownian motion.
	\newblock {\em Collect. Math.}, 53(1):55--81, 2002.
	
	\bibitem{Pazy}
	A. Pazy.
	\newblock{\em Semigroups of Linear Operators and Applications to Partial Differential Equations}.
	\newblock{Springer Applied Mathematical Series. Springer--Verlag, Berlin, 1983}.
	\bibitem{Ryan}
	R.~A.~Ryan.
	\newblock{\em Introduction to tensor products of Banach spaces}.
	\newblock Springer Monographs in Mathematics, Springer, 2002.
	\bibitem{Scheutzow}
	M.~Scheutzow. 
	\newblock On the perfection of crude cocycles.
	\newblock {\em Random Comput. Dynam.}, 4(4):235--255, 1996.
	\bibitem{Yagi}
	A.~Yagi.
	\newblock {\em Abstract Parabolic Evolution Equations and their Applications}.
	\newblock Springer, 2010.
	
	
	\bibitem{Young}
	L.~C.~Young.
	\newblock{An integration of H\"older type, connected with Stieltjes integration}.
	\newblock{\em Acta Math}. 67(1):251--282, 1936.
	\bibitem{Zaehle}
	M.~Z\"ahle. 
	\newblock{Integration with respect to fractal functions and stochastic calculus I}. 
	\newblock{\em Probab.
		Theory Related Fields}. 111(3):333--374, 1998.
	\end{thebibliography}
\end{document}